\documentclass[a4paper,oneside]{amsart}
\usepackage[utf8]{inputenc}

\title[Strongly ergodic, distal actions and model theory]{A model-theoretic approach to rigidity of strongly ergodic, distal actions}
\author{Tom{\'a}s Ibarluc{\'\i}a}
\address{
  Institut de Math{\'e}matiques de Jussieu--PRG \\
  Universit\'e Paris Diderot, case 7012 \\
  8, place Aur{\'e}lie Nemours \\
  75205 Paris \textsc{cedex} 13 \\
  France}
\urladdr{https://webusers.imj-prg.fr/~tomas.ibarlucia}

\author{Todor Tsankov}
\address{
  Institut de Math{\'e}matiques de Jussieu--PRG \\
  Universit\'e Paris Diderot \\
  75205 Paris \textsc{cedex} 13 \\
  France
  -- and --
  D{\'e}partement de math{\'e}matiques et applications \\
  {\'E}cole normale sup{\'e}rieure \\
  75005 Paris \\
  France}
\curraddr{Institut Camille Jordan \\
  Universit\'e Claude Bernard Lyon 1 \\
  43, boulevard du 11 novembre 1918 \\
  69622 Villeurbanne \textsc{cedex} \\
  France}

\email{tsankov@math.univ-lyon1.fr}

\date{\today}

\usepackage[T1]{fontenc}
\usepackage[osf]{mathpazo}
\usepackage{amssymb}

\usepackage{eucal}

\usepackage{microtype}

\usepackage{hyperref}
\usepackage[bottom=3.5cm]{geometry}

\usepackage{enumitem}
\setlist[enumerate,1]{label=(\roman*), font=\normalfont}

\usepackage[initials, shortalphabetic]{amsrefs}

\usepackage{my-macros}

\newcommand{\FPMP}{\mathrm{FPMP}}
\newcommand{\PMP}{\mathrm{PMP}}
\newcommand{\APA}{\mathrm{APA}}
\newcommand{\meq}{\mathrm{meq}}

\DeclareMathOperator{\rk}{rk}
\DeclareMathOperator{\Cb}{Cb}
\DeclareMathOperator{\alg}{alg}

\DeclareMathOperator{\Eacl}{\acl_\exists}

\begin{document}

\begin{abstract}
  We develop a model-theoretic framework for the study of distal factors of strongly ergodic, measure-preserving dynamical systems of countable groups. Our main result is that all such factors are contained in the (existential) algebraic closure of the empty set. This allows us to recover some rigidity results of Ioana and Tucker-Drob as well as prove some new ones: for example, that strongly ergodic, distal systems are coalescent and that every two such systems that are weakly equivalent are isomorphic. We also prove the existence of a universal distal, ergodic system that contains any other distal, ergodic system of the group as a factor.
\end{abstract}
%%% Résumé en français : %%%
% Nous développons un cadre modèle-théorique pour l'étude de facteurs distaux d'actions fortement ergodiques de groupes dénombrables. Notre résultat principal est que lesdits facteurs sont contenus dans la clôture algébrique (existentielle) du vide. Cela nous permet de récupérer un théorème de rigidité de Ioana et Tucker-Drob ainsi que de démontrer de nouveaux résultats : par exemple, que les systèmes distaux fortement ergodiques sont coalescents et que si deux systèmes de ce type sont faiblement équivalents, alors ils sont isomorphes. Nous prouvons également l'existence d'un système distal ergodique universel qui contient tout système distal ergodique du groupe comme facteur.
%%%%%%%%%%%%%%%%%%%%%%%%%%%

\maketitle
\setcounter{tocdepth}{1}
\tableofcontents

\section{Introduction}
\label{sec:introduction}

The theory of \df{compact} measure-preserving dynamical systems was initiated by Halmos and von Neumann, who characterized the ergodic systems of the group of integers $\Z$ that can be represented as translations on a compact group in terms of their spectrum. This was extended by Mackey~\cite{Mackey1964} to general locally compact groups. Later, inspired by the work of Furstenberg in topological dynamics, Furstenberg~\cite{Furstenberg1977} and Zimmer~\cite{Zimmer1976, Zimmer1976a} considered the relative notion of a \df{compact extension} of a system and defined a \df{distal} system to be one obtained via a transfinite tower of compact extensions, starting from the trivial system. The notion of a distal system was central to Furstenberg's ergodic-theoretic proof of Szemerédi's theorem, and led to the \emph{Furstenberg--Zimmer structure theorem} for general ergodic actions of locally compact groups. In the literature, compact systems are often referred to as \df{isometric} or having \df{discrete} (or \df{pure point}) \df{spectrum}. Similarly, distal systems are also known as systems with \df{generalized discrete spectrum}. In this paper, we will use the terminology ``compact'' and ``distal''.

The notion of \df{weak containment} of measure-preserving systems of countable groups was introduced by Kechris~\cite{Kechris2010} as a weakening of the notion of a \df{factor}: a system $\cX$ is \df{weakly contained} in another system $\cY$ (notation: $\cX <_w \cY$) if $\cX$ is a factor of an ultrapower of $\cY$. Two systems $\cX$ and $\cY$ are \df{weakly equivalent} if $\cX <_w \cY$ and $\cY <_w \cX$. It turns out that weak equivalence is a \df{smooth} equivalence relation (the set of equivalence classes is compact \cite{Abert2011p}) and that many notions in ergodic theory are invariants of weak equivalence. Free actions of infinite amenable groups are all weakly equivalent, but non-amenable groups usually admit uncountably many classes of weak equivalence~\cite{Bowen2018}.

By a result of Tucker-Drob~\cite{Tucker-Drob2015a}, weak equivalence classes of free actions always contain uncountably many isomorphism classes (and they are not even classifiable by countable structures). This is why rigidity results that allow to recover the isomorphism class from the weak equivalence class of a system in special situations are particularly interesting. Two rigidity results about weak equivalence have appeared in the literature. The first one, due to Abért and Elek~\cite{Abert2012a}, states that if a \df{profinite} system $\cX$ is weakly contained in a strongly ergodic system $\cY$, then $\cX$ is a factor of $\cY$. This was then generalized by Ioana and Tucker-Drob~\cite{Ioana2016} to distal systems. From this, they were able to conclude that for \emph{compact}, strongly ergodic systems, weak equivalence implies isomorphism, because compact ergodic systems are \df{coalescent}, i.e., every endomorphism of the system is an automorphism. However, this is not true in general for distal systems (see Parry and Walters~\cite{Parry1970}). Nevertheless it turns out to be true for strongly ergodic, distal systems and we were able to prove the following.

\begin{theorem}
  \label{th:i:distal-rigidity}
  Let $\Gamma$ be a countable group.
  \begin{enumerate}
  \item Let $\cX$ be a distal, strongly ergodic, probability measure-preserving $\Gamma$-system. Then $\cX$ is coalescent and $\Aut(\Gamma\actson\cX)$ is compact.
  \item Let $\cX$ and $\cY$ be two distal, strongly ergodic, probability measure-preserving $\Gamma$-systems. If they are weakly equivalent, then they are isomorphic.
  \end{enumerate}
\end{theorem}

In the course of the proof of Theorem~\ref{th:i:distal-rigidity}, we also give a new proof of the rigidity result of Ioana and Tucker-Drob mentioned above. The methods we use come from continuous logic, a relatively new branch of model theory, suitable for studying metric structures. While some model-theoretic concepts, such as ultraproducts, have been used before in ergodic theory, to our knowledge, this is the first application of continuous logic to dynamics of countable groups.

The main model-theoretic notion behind the proof is that of \df{algebraic closure}: if $M$ is a metric structure, $A \sub M$, and $b \in M$, we say that $b$ is in the \df{algebraic closure} of $A$ (notation: $b \in \acl^M(A)$) if it belongs to a compact set definable over $A$. Our main model-theoretic result can be stated as follows.
\begin{theorem}
  \label{th:i:acl}
  Let $\cX$ be a strongly ergodic system, $\cZ$ be a factor of $\cX$ and $\cY$ be a distal extension of $\cZ$ in $\cX$. Then $\cY \sub \acl^\cX(\cZ)$.
\end{theorem}
A slight strengthening of Theorem~\ref{th:i:acl} (replacing $\acl$ by $\Eacl$) easily implies Theorem~\ref{th:i:distal-rigidity}. We provide two rather different proofs of Theorem~\ref{th:i:acl}, one based on the theory of compact extensions, and another using model-theoretic stability theory.

Another feature of our approach is that it relies, in part, on the existence of certain canonical universal systems. Those systems are usually non-separable, i.e., they cannot be realized on a standard probability space. Working on a standard space is an assumption often made in the literature, either out of habit, or for the more essential reason that many important results fail in the non-separable setting (we give several examples in this paper). We depart from this tradition and \emph{make no separability assumptions throughout the paper unless specifically stated}. A posteriori, somewhat surprisingly, it turns out that  all strongly ergodic, distal systems are separable (Corollary~\ref{c:distal-is-in-Xalg}).

It is possible to carry out the stability-theoretic proof presented in Section~\ref{sec:two-line-proof} by only using ergodic theoretic results in a separable situation.

The following is our main result regarding universal systems.
\begin{theorem}
  \label{th:i:universal-distal}
  Let $\Gamma$ be a countable group and $\alpha$ be an ordinal. Then there exists a unique ergodic $\Gamma$-system $\cD_\alpha(\Gamma)$ of distal rank at most $\alpha$ which contains every ergodic $\Gamma$-system of distal rank at most $\alpha$ as a factor. Moreover, $\cD_{\omega_1}(\Gamma) = \cD_{\omega_1+1}(\Gamma)$, that is, $\cD_{\omega_1}(\Gamma)$ is the universal ergodic, distal $\Gamma$-system.
\end{theorem}
An analogous result bounding the rank of a \emph{topological} distal minimal system for $\Gamma = \Z$ was proved by Beleznay and Foreman~\cite{Beleznay1995}.

For $\alpha = 1$, the existence of a universal system is well-known: the maximal ergodic, compact system of $\Gamma$ is a translation on a compact group known as the \df{Bohr compactification of $\Gamma$}. For many groups (for example, $\Z$), the Bohr compactification is not metrizable, i.e., the corresponding dynamical system is not separable. However, it is a result of Wang~\cite{Wang1975} that if $\Gamma$ has property (T), then its Bohr compactification is a metrizable group. We recover this result as a consequence of Theorem~\ref{th:i:acl} and prove something more: the hierarchy of distal systems for property (T) groups collapses.
\begin{theorem}
  \label{th:i:propT}
  Let $\Gamma$ have property (T) and $\cZ$ be an ergodic $\Gamma$-system. Then every ergodic, distal extension of $\cZ$ is a compact extension.
\end{theorem}
This theorem had been previously proved by Chifan and Peterson (unpublished) but our proof is different and independent from theirs. We would like to thank them for telling us about their result.

For $\Gamma = \Z$, Beleznay and Foreman~\cite{Beleznay1996} have proved that there exist separable ergodic, distal systems of rank $\alpha$ for every $\alpha < \omega_1$. This implies that the tower of universal distal systems $(\cD_\alpha(\Z))_{\alpha \leq \omega_1}$ is proper.

\begin{question}
  \label{q:rank-distal}
  Let $\Gamma$ be a countable group and suppose that $\cD_2(\Gamma)$ is not strongly ergodic. Is it true that $\cD_\alpha(\Gamma) \subsetneq \cD_{\alpha+1}(\Gamma)$ for all $\alpha < \omega_1$?
\end{question}
In Theorem~\ref{th:propT} we prove that if $\cD_2(\Gamma)$ is strongly ergodic, then every distal system of $\Gamma$ is compact, so the condition that $\cD_2(\Gamma)$ is not strongly ergodic is necessary.

Until recently, it was an open question whether strongly ergodic, distal, non-compact systems exist. It was resolved by Glasner and Weiss~\cite{Glasner2018p} who, using results of Bourgain and Gamburd on compact systems with spectral gap, constructed examples for $\Gamma$ a free group. 

Finally, we would like to mention that the assumption that $\Gamma$ is countable is made throughout the paper for simplicity. All results with appropriate modifications (for example replacing $\omega_1$ by $|\Gamma|^+$) hold for arbitrary \emph{discrete} groups $\Gamma$. On the other hand, for the moment, our methods do not allow us to treat non-discrete, locally compact groups.

The paper is organized as follows. In Section~\ref{sec:model-theor-fram}, we describe how to treat probability measure-preserving dynamical systems in a model-theoretic framework and study the connection between weak containment and the existential theory. In Section~\ref{sec:basic-facts-cpct-exts}, we develop some of the theory of compact extensions without separability assumptions. In Section~\ref{sec:universal-distal-systems}, we consider universal systems and prove Theorem~\ref{th:i:universal-distal}. In Sections~\ref{sec:comp-extens-algebr} and \ref{sec:two-line-proof} we give two proofs of Theorem~\ref{th:i:acl} and derive Theorem~\ref{th:i:distal-rigidity} from it. Finally, in Section~\ref{sec:groups-with-property-T}, we consider groups with property~(T). In an appendix, we include a proof of the fact that the distal rank does not increase on factors.

\subsection*{Acknowledgments}
We are grateful to Matt Foreman, Eli Glasner, François Le Maître and Robin Tucker-Drob for useful discussions. Research was partially supported by the ANR projects AGRUME (ANR-17-CE40-0026) and GAMME (ANR-14-CE25-0004).

%%%%%%%%%%%%%%%%%%%%%%%%%%%%%%%%%%%%%%%%%%%%%%%%%%

\section{A model-theoretic framework for ergodic theory}
\label{sec:model-theor-fram}

\subsection{The language and the basic theory}
\label{sec:lang-basic-theory}

In this section, we explain how to formalize probability measure-preserving (pmp) group actions within the framework of continuous logic. The structures of interest are measure-preserving group actions $\Gamma \actson (X, \cX, \mu)$ where $\Gamma$ is a discrete group and $(X, \cX, \mu)$ is a probability space. As continuous logic deals with metric structures, it is natural to consider the measure algebra $\MALG(X, \cX, \mu)$ obtained by taking the quotient of $\cX$ by the equivalence relation of having symmetric difference of measure $0$. We will abuse notation and write again $\cX$ for this quotient. It becomes naturally a complete metric space if endowed with the metric $d_\mu(a, b) = \mu(a \sdiff b)$. By duality, a measure-preserving action $\Gamma \actson (X, \mu)$ is nothing but an action $\Gamma \actson \cX$ by isometries that also preserves the algebra structure.

The signature that we consider will consist of the language of Boolean algebras $\set{\cap, \cup, \sminus, \emptyset, X}$ (here $\sminus$ is set difference and $X$ stands for the maximum element of the Boolean algebra), a $[0, 1]$-valued predicate for the measure $\mu$, and a unary function symbol for every element $\gamma \in \Gamma$.

For a metric structure $\cX$, it is natural to measure its size using the \df{density character} (or just \df{density}), i.e., the smallest cardinal of a dense subset of $\cX$. We will denote the density of $\cX$ by $|\cX|$. Note that if $\cX$ is an infinite measure algebra, then $|\cX| = |L^2(\cX)|$ and $|\cX| \leq \aleph_0$ iff $\cX$ is isomorphic to the measure algebra of a \df{standard probability space}, i.e., one that can be realized by a Borel measure on the interval $[0, 1]$.

\begin{remark}
  \label{rem:Linfty}
  Sometimes it will be useful to consider the richer von Neumann algebra $L^\infty(X, \cX, \mu)$. It is a metric structure in the language that includes sorts for all $\nm{\cdot}_\infty$-balls and symbols for addition, multiplication, the adjoint, and the $L^2$-norm (which defines the metric). Note that the $\nm{\cdot}_\infty$-norm is \emph{not} part of the language. It follows from \cite{BenYaacov2013a} that the structures $\cX$ and $L^\infty(X, \cX, \mu)$ are bi-interpretable, so we can use them interchangeably. We will denote the latter structure simply by $L^\infty(\cX)$. (Many alternative but equivalent presentations of the structure $L^\infty(\cX)$ are possible; in fact, the language chosen in \cite{BenYaacov2013a} is quite different from the one proposed here, for instance in that it is one-sorted---only $[0,1]$-valued functions are considered---and that the $L^1$-norm is used instead of the $L^2$-norm.)
\end{remark}

\begin{defn}
  \label{df:MPA_Gamma}
The theory $\PMP_\Gamma$ includes the following axioms:
\begin{itemize}
\item The universe is a probability measure algebra (see \cite{BenYaacov2008}*{Section~16} for the precise axioms).
\item Each $\gamma$ is an automorphism.
\item For every tuple $\gamma_1, \gamma_2, \ldots, \gamma_n$ of elements of $\Gamma$ such that $\gamma_1 \gamma_2 \cdots \gamma_n = 1_\Gamma$,
  % the axiom $\sup_a d_\mu\bigl(\gamma_1 \cdots \gamma_n a, a \bigr) = 0$.
  the automorphism $\gamma_1 \gamma_2 \cdots \gamma_n$ is equal to the identity.
\end{itemize}
\end{defn}

\begin{remark}
  \label{rem:only-generators}
  If $\Gamma$ is generated by a subset $S \sub \Gamma$, then we can include in the language only symbols for the elements of $S$ and replace the third group of axioms by the relations that define the group. For example, if $\Gamma = \bF_n$ is the free group on $n$ generators, we can include in the language $n$ function symbols and the third group of axioms will be empty.
\end{remark}

The models of $\PMP_\Gamma$ are exactly the pmp actions of $\Gamma$. If $\cX \models \PMP_\Gamma$ and $\cY$ is a substructure of $\cX$, then $\cY$ is a closed, $\Gamma$-invariant subalgebra of $\cX$, or, which is the same, a \df{factor} of $\cX$. We will use the words \emph{substructure} and \emph{factor} interchangeably and we will use the notation $\cY \sub \cX$. Dually, if $\iota \colon \cY \to \cX$ is an embedding, there exist spatial representations $(X, \mu)$ and $(Y, \nu)$ of $\cX$ and $\cY$ and a $\Gamma$-equivariant, measurable map $\pi \colon X \to Y$ such that $\pi_* \mu = \nu$ and $\iota(a) = \pi^{-1}(a)$ for all $a \in \cY$. The fact that substructures of models of $\PMP_\Gamma$ are again models implies that $\PMP_\Gamma$ admits a \df{universal axiomatization} (that is, with axioms of the form $\sup_{\bar a} \phi(\bar a) \leq 0$, where $\phi$ is quantifier-free), something that can also be verified directly (if we work only with symbols for a generating set $S$, as in Remark~\ref{rem:only-generators}, then $S$ should be symmetric).

Often one is interested in \df{free actions}, i.e., actions for which non-identity elements of the group do not fix any points in $X$.
\begin{defn}
  \label{df:FREE_Gamma}
  The theory $\FPMP_\Gamma$ includes all axioms of $\PMP_\Gamma$ and, in addition, for every element $\gamma \in \Gamma$ of finite order $n$, the axiom
  \begin{itemize}
  \item $\inf_{\set{a : \mu(a) = 1/n}} \max_{i < j < n} \mu(\gamma^i a \cap \gamma^j a) = 0$,
  \end{itemize}
  and for every element $\gamma \in \Gamma$ of infinite order, an axiom of the previous form for every $n> 1$.
\end{defn}
It is clear that any action satisfying those axioms is free and, conversely, any free action satisfies the axioms (this follows from Rokhlin's lemma for $\gamma$ of infinite order). If $\Gamma$ is infinite, the freeness of the action implies that the measure space is non-atomic, so separable models of $\FPMP_\Gamma$ are, up to isomorphism, exactly the free, measure-preserving actions $\Gamma \actson ([0, 1], \lambda)$, where $\lambda$ is the Lebesgue measure.

Let us mention a few known facts about the model theory of $\FPMP_\Gamma$. If $\Gamma$ is finite, we need to add non-atomicity as an axiom but once we do, the theory that we obtain is \emph{$\omega$-categorical}, i.e., it has a unique separable model up to isomorphism. Indeed, given $\Gamma\actson X_1$ and $\Gamma\actson X_2$, we may choose fundamental domains $A_i\subseteq X_i$ (i.e., measurable subsets meeting every orbit in exactly one point), then extend any measure-preserving isomorphism $A_1\simeq A_2$ to an equivariant pmp isomorphism $X_1\simeq X_2$. Conversely, if $\Gamma$ is infinite, the theory $\FPMP_\Gamma$ is not $\omega$-categorical.

If the group $\Gamma$ is infinite and amenable, $\FPMP_\Gamma$ is a complete, stable theory that eliminates quantifiers. This was proved in \cite{BenYaacov2008} for $\Gamma = \Z$ and was generalized in \cite{Berenstein2018p} to all amenable~$\Gamma$.

\subsection{The existential theory and weak containment}
\label{sec:exist-theory-weak}

The notion of \df{weak containment} for measure-preserving actions was introduced by Kechris~\cite{Kechris2010}, who was motivated by a similar notion for unitary representations. This notion has a strong model-theoretic flavor and it turns out that it can be expressed naturally in terms of the existential theories of the actions.

\begin{defn}
  \label{df:wcon}
  Let $\Gamma$ be a discrete group and let $\cX$ and $\cY$ be two measure-preserving actions of $\Gamma$. We say that $\cX$ is \df{weakly contained} in $\cY$ (notation: $\cX <_w \cY$) if for every $\epsilon>0$ and finitely many $a_0,\dots,a_{n-1}\in\cX$ and $\gamma_0,\dots,\gamma_{k-1} \in \Gamma$ there are $b_0,\dots,b_{n-1} \in \cY$ such that $|\mu(a_i\cap\gamma_l a_j)-\mu(b_i\cap\gamma_l b_j)|\leq\epsilon$ for every $i,j<n$ and $l<k$.
\end{defn}
In the literature, the symbol $\prec$ is often used to represent weak containment; however, this conflicts with the well-established notation for an elementary substructure in model theory.

Recall that an \df{$\inf$-sentence} is a sentence of the form $\inf_{\bar a} \phi(\bar a)$, where $\phi$ is quantifier-free. The \emph{existential theory} of a structure $M$, denoted by $\Th_\exists(M)$, is given by the set of statements of the form $\phi \leq 0$ true in $M$, where $\phi$ is an $\inf$-sentence. It is clear that if $\cX$ and $\cY$ are $\Gamma$-systems with $\Th_\exists(\cX) \sub \Th_\exists(\cY)$, then $\cX <_w \cY$. The converse is also easy to see. On the other hand, by the compactness theorem for continuous logic, the condition $\Th_\exists(M) \sub \Th_\exists(N)$ is equivalent to saying that $M$ embeds in an elementary extension of $N$. We note these observations down; the equivalence between the first and third item below was first proved by Conley, Kechris and Tucker-Drob in \cite{Conley2013a} (see also Carderi~\cite{Carderi2015}).

\begin{lemma}
  \label{l:wcon-existential}
  Let $\cX$ and $\cY$ be two pmp\ actions of $\Gamma$. Then the following are equivalent:
  \begin{enumerate}
  \item $\cX$ is weakly contained in $\cY$.
  \item $\Th_\exists(\cX) \sub \Th_\exists(\cY)$ or, which is the same, $\cY \models \Th_\exists(\cX)$.
  \item There is an elementary extension (equivalently, an ultrapower) $\cY'$ of $\cY$ such that $\cX\subseteq\cY'$.
  \end{enumerate}
\end{lemma}
Thus two systems are \df{weakly equivalent} if they have the same existential theories. Alternatively, two systems $\cX$ and $\cY$ are weakly equivalent iff for all $\inf$-sentences $\phi$, we have that $\phi^\cX = \phi^\cY$. The space $W$ of weak equivalence classes of actions of $\Gamma$ can be identified with the space of existential theories consistent with $\PMP_\Gamma$. The topology on $W$ is the weakest topology that makes evaluation of $\inf$-sentences on $W$ continuous; it follows from the compactness theorem that this topology is compact as $W$ is a quotient of the space of all complete theories compatible with $\PMP_\Gamma$. This recovers a result of Abért and Elek~\cite{Abert2011p}. Note that $W$ can also be viewed as the Gelfand space of the abstract C$^*$-algebra generated by the $\inf$-sentences modulo $\PMP_\Gamma$.

\subsection{Strongly ergodic actions}

A system $\cX$ is \emph{ergodic} if every invariant set is trivial or, which is the same, if $\cX$ does not contain any trivial two-point system as a factor. In the same vein, a system $\cX$ is \df{strongly ergodic} if almost invariant sets are close to $\emptyset$ or $X$ or, more precisely, if it does not weakly contain any trivial two-point system. Equivalently, $\cX$ is strongly ergodic if for every $\epsilon>0$ there are $\delta>0$ and finitely many $\gamma_0, \dots, \gamma_{n-1} \in \Gamma$ such that, whenever $\epsilon < \mu(a) < 1-\epsilon$, we have $\max_{i<n}d_\mu(a, \gamma_ia)\geq\delta$.
If $\Gamma$ is finitely generated, then in the previous definition one can replace the $\gamma_i$ by the generators of $\Gamma$.

We have the following additional easy equivalences, which indicate why model-theoretic methods are likely to be useful for the study of these actions.

\begin{prop}
Let $\cX$ be a pmp\ $\Gamma$-system. The following are equivalent.
\begin{enumerate}
\item $\cX$ is strongly ergodic.
\item Every model of the (universal) first-order theory of $\cX$ is ergodic.
\item Every elementary extension $\cX'$ of $\cX$ is ergodic.
\end{enumerate}
\end{prop}

A group $\Gamma$ has \emph{property~(T)} if every ergodic pmp  $\Gamma$-system is strongly ergodic. (This is not the standard definition but it is equivalent by \cite{Connes1980, Schmidt1981}.) Below we give a characterization of property (T) in terms of the theory $\PMP_\Gamma$.

Recall that in continuous logic, a set $S$ in a given structure is \emph{$0$-definable} if the distance function to $S$ is a definable predicate. Similarly, if $S$ denotes a family of sets, one in each model of a given theory $T$, then $S$ is \emph{uniformly $0$-definable} if there is a common formula that gives in each model of $T$ the distance function to the corresponding set in $S$.

\begin{prop}
  \label{p:T-definable}
  Let $\Gamma$ be a countable group. Then the following are equivalent:
  \begin{enumerate}
  \item \label{i:Tdef:1} $\Gamma$ has property (T).
  \item \label{i:Tdef:2} The algebra of invariant sets $\set{a \in \cX : \gamma a = a \text{ for all } \gamma \in \Gamma}$ is uniformly $0$-definable in the models of $\PMP_\Gamma$.
  \item \label{i:Tdef:3} The class of ergodic $\Gamma$-actions is axiomatizable.
  \end{enumerate}
\end{prop}
\begin{proof}
  \begin{cycprf}
  \item[\impnext] We use the equivalences of definability of sets given in \cite{BenYaacov2008}*{Proposition~9.19}, which also hold uniformly. Then the implication follows from the well-known fact that for a property (T) group $\Gamma$ with a finite generating set $Q$, for every $\eps > 0$ there exists $\delta > 0$ such that for any pmp action $\Gamma \actson \cX$, if $a \in \cX$ is $(Q, \delta)$-invariant (i.e., if $\max_{\gamma\in Q}d_\mu(a,\gamma a)\leq \delta$), then there exists $a' \in \cX$ which is $\Gamma$-invariant and satisfies $d_\mu(a, a') \leq \eps$.
    
  \item[\impnext] If $P(a)$ denotes the distance of $a$ to the algebra of invariant sets, then the condition
  $\sup_a \big(\min(\mu(a), 1-\mu(a)) - P(a)\big) \leq 0$ axiomatizes the class of ergodic systems within the models of $\PMP_\Gamma$.
    
  \item[\impfirst] If $\cX$ is a an ergodic $\Gamma$-system which is not strongly ergodic, then any saturated elementary extension of $\cX$ is not ergodic but is elementarily equivalent to $\cX$.
  \end{cycprf}
\end{proof}

%%%%%%%%%%%%%%%%%%%%%%%%%%%%%%%%%%%%%%%%%%%%%%%%%%

\section{Some basic facts about compact extensions and distal systems}
\label{sec:basic-facts-cpct-exts}

First we recall some of the theory of compact systems and compact extensions, as developed by Mackey~\cite{Mackey1966}, Furstenberg~\cite{Furstenberg1977}, and Zimmer~\cite{Zimmer1976a} (see also Glasner~\cite{Glasner2003}). In the above references, the authors work under the assumption that the systems are separable and make use of a variety of decomposition results that are only valid for separable spaces. As we work in a more general setting, we prove all results that we need.

Let $\cZ\subseteq\cX$ be an extension of measure algebras. A Hilbert subspace $M\subseteq L^2(\cX)$ is a \emph{$\cZ$-module} if it is closed under multiplication by functions of $L^\infty(\cZ)$. The $\cZ$-module generated by a set $S\subseteq L^2(\cX)$ is the closed span of $L^\infty(\cZ)S$.

Given $f,g\in L^2(\cX)$, we define $\langle f,g\rangle_\cZ\coloneqq \E(f\bar g|\cZ)\in L^1(\cZ)$ (the conditional expectation of $f\bar g$ relative to $\cZ$). A set $S\subseteq L^2(\cX)$ is \emph{$\cZ$-orthonormal} if $\langle f,g\rangle_\cZ=0$ for distinct $f,g\in S$, and $\langle f,f\rangle_\cZ$ is a characteristic function for each $f\in S$. If $f \in L^2(\cX)$, its \df{$\cZ$-support} (denoted by $\supp_\cZ f$) is the support of the function $\ip{f, f}_\cZ$, which is an element of $\cZ$. We will write $\chi_\cZ(f)$ for the characteristic function of $\supp_\cZ f$. Say that $S$ is a \emph{$\cZ$-basis} for a $\cZ$-module $M$ if it is $\cZ$-orthonormal, generates $M$, and each element of $S$ has full $\cZ$-support. 

The following lemma summarizes several properties of the conditional inner product $\ip{\cdot, \cdot}_\cZ$.
\begin{lemma}
  \label{l:inner-prod-Z}
  Let $\set{e_i : i < n}$ be a finite, generating, $\cZ$-orthonormal set for a $\cZ$-module $M \sub L^2(\cX)$, and let $f, g \in M$. Then the following statements hold:
  \begin{enumerate}
  \item \label{i:ip:1} (Cauchy--Schwarz) $|\ip{f, g}_\cZ|^2 \leq \ip{f, f}_\cZ \ip{g, g}_\cZ$.
  \item \label{i:ip:2} $\supp \ip{f, g}_\cZ \sub \supp_\cZ f \cap \supp_\cZ g$. In particular,
    \begin{equation*}
      \ip{f, g}_\cZ = \ip{f, g}_\cZ \chi_\cZ(f) \chi_\cZ(g).
    \end{equation*}
  \item \label{i:ip:4} $f = \sum_{i < n} \ip{f, e_i}_\cZ e_i$.
  \end{enumerate}
\end{lemma}
\begin{proof}
  \ref{i:ip:1} is standard and \ref{i:ip:2} follows from \ref{i:ip:1}.
  For \ref{i:ip:4}, first we check that $\ip{f, e_i}_\cZ e_i \in L^2(\cX)$. Using \ref{i:ip:1},
  \begin{equation*}
    \int |\ip{f, e_i}_\cZ e_i|^2 \leq \int \ip{f, f}_\cZ |e_i|^2 \leq \nm{f}^2.   
  \end{equation*}
  We also have $\ip{e_i, e_i}_\cZ e_i = e_i$.
  This implies that the identity holds when $f = \sum_{i < n} a_i e_i$ with $a_i \in L^\infty(\cZ)$, and as it is a closed condition, it must hold for all $f \in M$.
\end{proof}

\begin{lemma}
  \label{l:os-size}
  Let $\set{e_i : i < n}$ and $\set{f_j : j < m}$ be two generating, $\cZ$-orthonormal sets for a $\cZ$-module $M$. Let $a_{ij} = \ip{e_i, f_j}_\cZ$ and $A = (a_{ij})_{i < n, j < m}$. Then the following hold:
  \begin{enumerate}
  \item The entries of $A$ are in $L^\infty(\cZ)$.
  \item $AA^* = \diag \big(\chi_\cZ(e_0), \ldots, \chi_\cZ(e_{n-1}) \big)$ and $A^*A = \diag \big(\chi_\cZ(f_0), \ldots, \chi_\cZ(f_{m-1}) \big)$.
  \item $\sum_{i < n} \ip{e_i, e_i}_\cZ = \sum_{j < m} \ip{f_j, f_j}_\cZ = \sum_{i, j} |a_{ij}|^2$.
  \end{enumerate}
\end{lemma}
\begin{proof}
  This follows easily from Lemma~\ref{l:inner-prod-Z} and standard linear algebra.
\end{proof}

In the context of $\Gamma$-systems, we will be mainly interested in $\cZ$-modules that are $\Gamma$-invariant and finitely generated.

\begin{lemma}
Let $\cX$ be a $\Gamma$-system and let $\cZ$ be an ergodic factor thereof. Then, every finitely generated $\Gamma$-invariant $\cZ$-module $M\subseteq L^2(\cX)$ admits a finite $\cZ$-basis.
\end{lemma}
\begin{proof}
Let $f_1, \ldots, f_n$ be generators for $M$. We use Gram--Schmidt orthogonalization. Set
\begin{equation*}
  h_1 = f_1, \quad h_2 = f_2 - P_{\langle f_1 \rangle_{L^\infty(\cZ)}} f_2, \quad \ldots, \quad h_n = f_n - P_{\langle f_1, \ldots, f_{n-1} \rangle_{L^\infty(\cZ)}} f_n,
\end{equation*}
where $P_{\langle F \rangle_{L^\infty(\cZ)}}$ denotes the orthogonal projection onto the $\cZ$-module generated by $F$. Then $\langle h_i,h_j\rangle_\cZ=0$ for $i\neq j$ and $M = \boplus_{i = 1}^n M_{h_i}$ where $M_h$ is the $\cZ$-module generated by $h$. Let $S_i = \supp_\cZ h_i$. We define
$e_i=\big(\langle h_i,h_i\rangle_\cZ\big)^{-1/2}h_i$ with the convention that $e_i$ vanishes outside $S_i$. Then $\int |e_i|^2=\int_{S_i}\frac{|h_i|^2}{\langle h_i,h_i\rangle_\cZ}=\mu(S_i)$, so $e_i \in L^2(\cX)$ and $\set{e_i : i < n}$ is a generating $\cZ$-orthonormal system for $M$. (Note that if $h\in L^2(\cX)$ and $g\in L^0(\cZ)$ is any $\cZ$-measurable function such that $gh\in L^2(\cX)$, then $gh\in M_h$.)

By Lemma~\ref{l:os-size}, the function $\sum_{i<n}\bOne_{S_i}$ is $\Gamma$-invariant, so by the ergodicity assumption it must be constant, say equal to $k$. Denote by $\binom{n}{k}$ the set of all $k$-element subsets of $n$. For $\alpha \in \binom{n}{k}$, let $A_\alpha = \bigcap_{i \in \alpha} S_i$.
For $i < k$, define $\sigma_i(\alpha)$ for $\alpha \in \binom{n}{k}$ to be the $i$-th element of $\alpha$. Finally, for $i < k$, define $e_i' = \sum_{\alpha \in \binom{n}{k}} \bOne_{A_\alpha} e_{\sigma_i(\alpha)}$. Then one easily checks that $\set{e_i' : i < k}$ is a $\cZ$-basis for $M$.
\end{proof}

Suppose now that $\cZ\subseteq\cX$ is an extension of $\Gamma$-systems and that $M$ is a $\Gamma$-invariant $\cZ$-module admitting a $\cZ$-basis $\{e_0,\dots,e_{n-1}\}$. Let $X$ be the probability space associated to $\cX$. If $\gamma\in\Gamma$, the set $\{\gamma e_i\}_{i<n}$ is again a $\cZ$-basis for $M$. Thus, if for $x\in X$ we define
\begin{equation*}
u_\gamma(x)=\big(\langle\gamma e_i,e_j\rangle_{\cZ}(x)\big)_{ij} \in \C^{n\times n},  
\end{equation*}
by Lemma~\ref{l:os-size}, $u_\gamma(x)$ is a unitary matrix almost surely and the function $u_\gamma$ is $\cZ$-measurable.

Now consider the function $\theta\colon X\to \R^+$ defined by
\begin{equation*}
  \theta(x)=\sum_{i<n}|e_i(x)|^2=\langle E(x),E(x)\rangle,
\end{equation*}
where $E=(e_1,\dots,e_n)$ and $\ip{\cdot, \cdot}$ denotes the usual inner product on $\C^n$. We have, almost surely,
\begin{equation*}
  \theta(\gamma^{-1}x)=\langle (\gamma E)(x),(\gamma E)(x)\rangle=\langle u_\gamma(x)E(x),u_\gamma(x)E(x)\rangle=\langle E(x),E(x)\rangle.
\end{equation*}
That is, $\theta$ is $\Gamma$-invariant. Hence, if we further assume that $\cX$ is ergodic, $\theta$ must be constant, and this implies that each $e_i$ is bounded. Thus we have proved the following.

\begin{prop}
  \label{p:exist-basis}
  Let $\cX$ be an ergodic $\Gamma$-system, $\cZ \sub \cX$ be a factor, and $M \sub L^2(\cX)$ be a finitely generated, $\Gamma$-invariant $\cZ$-module. Then $M$ admits a finite $\cZ$-basis and the elements of any such basis are in $L^\infty(\cX)$.
\end{prop}

Recall that a system $\cX$ is called \df{compact} if $L^2(\cX)$ decomposes as a sum of finite-dimensional representations of $\Gamma$. An extension $\cX \supseteq \cZ$ is \df{compact} if $L^2(\cX)$ decomposes as a direct sum of finitely generated, $\Gamma$-invariant $\cZ$-modules. A factor generated by compact factors is a compact factor as well. More generally, the following holds.

\begin{prop}
  \label{p:factor-generated-by-compact-factors}
Let $\cX$ be an ergodic system. Suppose $\cZ_i\subseteq\cY_i\subseteq\cX$, $i \in I$, are factors such that each extension $\cZ_i\subseteq\cY_i$ is compact. Then the factor $\cY \sub \cX$ generated by all the $\cY_i$ is a compact extension of the factor $\cZ$ generated by all the $\cZ_i$.
\end{prop} 
\begin{proof}
 By hypothesis, we have $L^2(\cY_i)=\bigoplus_{M\in\mathcal{M}_i}M$ where each $M$ is a finitely generated $\Gamma$-invariant $\cZ_i$-module. If $\{e_0,\dots,e_{n-1}\}\subset L^\infty(\cX)$ is a $\cZ_{i_1}$-basis for $M_1\in\mathcal{M}_{i_1}$ and $\{f_0, \dots, f_{m-1}\} \subset L^\infty(\cX)$ is a $\cZ_{i_2}$-basis for $M_2\in\mathcal{M}_{i_2}$, we denote by $M_1M_2$ the $\cZ$-module generated by the products $\{e_p f_q : p <  n, q < m\}$ (which are in $L^\infty(\cX)$ by Proposition~\ref{p:exist-basis}). Similarly, we define the product modules $M_1M_2\cdots M_k$ over $\cZ$ for any given $M_j\in\mathcal{M}_{i_j}$. These are clearly $\Gamma$-invariant. Then it is easy to check that
 \begin{equation*}
   L^2(\cY)=\sum\{M_1\cdots M_k: k \in \N, M_j\in\mathcal{M}_{i_j}\text{ for some }i_j \in I\}.
 \end{equation*}
 A standard orthogonalization procedure using Zorn's lemma then yields $L^2(\cY)=\bigoplus_s M'_s$ for some appropriate finitely generated $\Gamma$-invariant $\cZ$-modules $M'_s$.
\end{proof} 

\begin{cor}
  \label{c:maximal-cpct-extension}
  Let $\cX$ be an ergodic system and $\cZ \sub \cX$ be a factor of $\cX$. Then there exists a factor $\cY$ of $\cX$ such that $\cZ \sub \cY \sub \cX$, the extension $\cY \supseteq \cZ$ is compact, and every compact extension of $\cZ$ in $\cX$ is contained in $\cY$. This $\cY$ is called \df{the maximal compact extension of $\cZ$ in $\cX$}.
\end{cor}

Next we check that a factor of a compact extension is also compact.

\begin{lemma}\label{l:factor-of-compact-is-compact}
Suppose we have $\cZ\subseteq \cY\subseteq \cX$ and the extension $\cZ\subseteq \cX$ is compact. Then the extensions $\cZ\subseteq \cY$ and $\cY\subseteq\cX$ are compact.
\end{lemma}
\begin{proof}
The projection of a finitely generated $\Gamma$-invariant $\cZ$-module in $L^2(\cX)$ to $L^2(\cY)$ is again a finitely generated $\Gamma$-invariant $\cZ$-module. This readily implies that $\cZ\subseteq \cY$ is compact. For the extension $\cY \sub \cX$, this is clear.
\end{proof}

Suppose that $N$ is a finitely generated $\Gamma$-invariant $\cZ$-module and suppose that $E = (e_0, \ldots, e_{n-1})$ is a $\cZ$-basis for $N$. Define the \df{matrix coefficients} $m_{E} \colon \Gamma \to M_n(L^\infty(\cZ))$ by
\begin{equation*}
  m_E(\gamma) = \big( \ip{\gamma \cdot e_i, e_j}_\cZ \big)_{i,j < n}.
\end{equation*}
\begin{lemma}
  \label{l:dim-at-most-n}
  Suppose $\cX \supseteq \cZ$ is a compact extension with $\cX$ ergodic. Let $E=(e_1,\dots,e_n)$ be a $\cZ$-orthonormal basis for a finitely generated, $\Gamma$-invariant module $M \sub L^2(\cX)$. Then the set $\{E'\in L^2(\cX)^{\oplus n} : m_{E'} = m_E \}$ spans a subspace of $L^2(\cX)^{\oplus n}$ of dimension at most~$n$.
\end{lemma}
\begin{proof}
  Let $E'=(e'_1,\dots,e'_n)$ be a tuple with $m_{E'} = m_E$. We claim that the function $x \mapsto \ip{E(x), E'(x)}$ is $\Gamma$-invariant and therefore constant. (Here, again, $\ip{\cdot, \cdot}$ denotes the usual inner product in $\C^n$.) Indeed, given $\gamma\in\Gamma$ and $i,j < n$, the conditional expectations $\ip{\gamma e_i,e_j}_\cZ$ and $\ip{\gamma e'_i,e'_j}_\cZ$ are equal. Hence, the $\cZ$-measurable map $u_\gamma\colon X\to U(n)$ given by $u_\gamma(x)=(\ip{\gamma e_i,e_j}_\cZ(x))_{ij}$ satisfies $(\gamma E)(x)=u_\gamma(x)E(x)$ and also $(\gamma E')(x)=u_\gamma(x)E'(x)$. Thus
  \begin{equation*}
  \langle E(\gamma^{-1}x),E'(\gamma^{-1}x)\rangle=\langle u_\gamma(x)E(x),u_\gamma(x)E'(x)\rangle =\langle E(x),E'(x)\rangle,  
\end{equation*}
as desired.

Now consider $n+1$ tuples $E_0, \ldots, E_n$ with $m_{E_i} = m_E$.
Consider the matrix $\bigl( \ip{E_i(x), E_j(x)} \bigr)_{i, j \leq n}$, which, by the previous claim, does not depend on $x$. It has rank at most $n$, so there exist constants $\lambda_0, \ldots, \lambda_n$ not all equal to $0$ such that for all $j$ and almost every $x$, $\sum_{j = 0}^n \lambda_i \ip{E_i(x), E_j(x)} = 0$. This implies that the function $\sum_i \lambda_i E_i \in L^2(X)^{\oplus n}$ is orthogonal to each $E_j$, so equal to $0$. Hence the $E_i$ are linearly dependent.
\end{proof}

If $\cX$ is a $\Gamma$-system, we will denote by $\Aut(\cX)$ the group of automorphisms of~$\cX$, i.e., the group of all automorphisms of the measure algebra that commute with the action of $\Gamma$. We equip $\Aut(\cX)$ with the pointwise convergence topology on $\cX$, which makes it into a topological group. $\Aut(\cX)$ can also be viewed as a closed subgroup of the unitary group $U(L^2(\cX))$ equipped with the strong operator topology. If $\cZ \sub \cX$ is a factor, we will denote
\begin{equation*}
  \Aut_\cZ(\cX) = \set{\sigma \in \Aut(\cX) : \sigma(a) = a \text{ for all } a \in \cZ}.
\end{equation*}
$\Aut_\cZ(\cX)$ is a closed subgroup of $\Aut(\cX)$.

\begin{prop}
  \label{p:compact-aut-group}
If $\cX \supseteq \cZ$ is a compact extension with $\cX$ ergodic, then $\Aut_\cZ(\cX)$ is compact.
\end{prop}
\begin{proof}
Let $M \sub L^2(\cX)$ be a finitely generated, $\Gamma$-invariant $\cZ$-module. Let $E=(e_1,\dots,e_n)$ be a $\cZ$-basis for $M$. Then, by Lemma~\ref{l:dim-at-most-n}, the orbit $\Aut_\cZ(\cX) \cdot E$ is contained in a finite-dimensional subspace of $L^2(\cX)^{\oplus n}$. This implies that the orbit of every $f \in L^2(\cX)$ is contained in a finite-dimensional subspace of $L^2(\cX)$, so $\Aut_{\cZ}(\cX)$, being a closed subgroup of a product of finite-dimensional unitary groups, is compact.
\end{proof}

By an isomorphism between $\Gamma$-invariant $\cZ$-modules we mean a linear isometry that commutes with the action of $\Gamma$ and multiplication by elements of $L^\infty(\cZ)$. Note that an isomorphism between $\cZ$-modules preserves the conditional inner product $\ip{\cdot, \cdot}_\cZ$. Note also that if $E$ is a $\cZ$-basis for a module $M$ and $E'$ is a $\cZ$-basis for a module $M'$, then there is an isomorphism between $M$ and $M'$ that maps $E$ to $E'$ iff $m_E = m_{E'}$. We have the following consequence of the previous lemma.

\begin{prop}
  \label{p:comp-ext-size-bound}
Suppose $\cX \supseteq \cZ$ is a compact extension with $\cX$ ergodic, and let $L^2(\cX)=\bigoplus_{i\in I}M_i$ be a decomposition into finitely generated, $\Gamma$-invariant $\cZ$-modules. Then, for each $i\in I$, the number of $\cZ$-modules appearing in the decomposition that are isomorphic to $M_i$ is bounded by the size of a $\cZ$-basis of $M_i$ and is therefore finite. In particular, $|\cX| \leq (|\cZ| + \aleph_0)^{\aleph_0}$.
\end{prop}
\begin{proof}
  Suppose for contradiction that $M_0, \ldots, M_n \sub L^2(\cX)$ are isomorphic invariant submodules of $L^2(\cX)$ that are pairwise orthogonal. Let $E \in L^2(\cX)^{\oplus n}$ be a $\cZ$-basis of $M_0$ and let $f_i \colon M_0 \to M_i$ be isomorphisms. Then $f_0(E), \ldots, f_n(E)$ are mutually orthogonal $\cZ$-bases of size $n$ that contradict Lemma~\ref{l:dim-at-most-n}.

  For the second assertion, denote $\kappa = |\cZ| + \aleph_0$ and observe that there are at most $\kappa^{\aleph_0}$ isomorphism classes of finitely generated, $\Gamma$-invariant $\cZ$-modules. Indeed, each such module is determined by its matrix coefficients $\ip{\gamma \cdot e_i, e_j}_\cZ \in L^\infty(\cZ)$ for some $\cZ$-basis $(e_0, \ldots, e_{n-1})$. As the cardinal of $L^\infty(\cZ)$ is~$\kappa^{\aleph_0}$, this proves the claim. As each module can appear only finitely many times in the decomposition of $L^2(\cX)$ and has density character $|\cZ|$, this shows that $|L^2(\cX)| \leq \kappa^{\aleph_0} \cdot \kappa = \kappa^{\aleph_0}$.
\end{proof}

We recall now that an extension is \emph{distal} if it can be obtained as a tower of compact extensions. More precisely, $\cZ\subseteq\cX$ is distal if there is an ordinal number $\alpha$ and intermediate extensions $\{\cX_\beta\}_{\beta\leq\alpha}$ such that
\begin{itemize}
\item $\cZ=\cX_0 \subseteq \cX_\beta \subseteq \cX_{\beta'}\subseteq \cX_\alpha=\cX$ for $\beta\leq\beta'\leq\alpha$,
\item each extension $\cX_\beta\subseteq\cX_{\beta+1}$ is compact,
\item if $\lambda\leq\alpha$ is a limit ordinal, then $\cX_\lambda$ is generated by  the factors $\cX_\beta$, $\beta<\lambda$.
\end{itemize}
We will say that $\{\cX_\beta\}_{\beta\leq\alpha}$ is a \emph{distal tower} for the extension $\cZ\subseteq\cX$. The minimal length $\alpha$ of such a tower is the \emph{rank} of the distal extension, which we denote by $\rk(\cX/\cZ)$. If $\cZ$ is the trivial system, then $\cX$ is a \emph{distal system}, and we denote the rank simply by $\rk(\cX)$.

As for compact extensions, we have the following.

\begin{prop}\label{p:factor-generated-by-distal-factors}
Let $\cX$ be an ergodic system and let $\cZ_i\subseteq\cY_i\subseteq \cX$, $i \in I$, be factors such that the extensions $\cZ_i\subseteq\cY_i$ are distal. Then the factor $\cY \sub \cX$ generated by all the $\cY_i$ is a distal extension of the factor $\cZ$ generated by all the $\cZ_i$. Moreover, $\rk(\cY/\cZ)\leq\sup_{i\in I}\rk(\cY_i/\cZ_i)$.

In particular, given an extension $\cZ\subseteq \cX$ with $\cX$ ergodic, there is a largest intermediate distal extension of $\cZ$ in $\cX$.
\end{prop}
\begin{proof}
Let $\alpha=\sup_{i\in I}\rk(\cY_i/\cZ_i)$ and, for each $i\in I$, choose a distal tower $\{\cY_{i,\beta}\}_{\beta\leq\alpha}$ for the extension $\cZ_i\subseteq \cY_i$. Given $\beta\leq\alpha$, let $\cY_\beta$ be the factor generated by all the $\cY_{i,\beta}$ for $i\in I$. Then, using Proposition~\ref{p:factor-generated-by-compact-factors}, it is clear that $\{\cY_\beta\}_{\beta\leq\alpha}$ is a distal tower for $\cZ\subseteq \cY$.
\end{proof}

\begin{cor}\label{c:bound-on-distal-rank}
If $\cZ\subseteq \cX$ is an ergodic, distal extension, then $\rk(\cX/\cZ)\leq (|\cZ|+\aleph_0)^+$.
\end{cor}
\begin{proof}
It is enough to apply the previous proposition to the extensions $\cZ\subseteq \cW$ where $\cW$ varies over the factors of $\cX$ generated by $\cZ$ together with an additional element of $\cX$. Indeed, each such factor $\cW$ has density character at most $\kappa=|\cZ|+\aleph_0$, hence $\rk(\cW/\cZ) < \kappa^+$. On the other hand, these factors generate~$\cX$.
\end{proof}

A deeper result of the Furstenberg--Zimmer theory is that factors of distal extensions are again distal. We will need this fact for arbitrary (not necessarily separable) systems. The following lemma will be useful to transfer this result (and some others) from the standard theory to the non-separable setting.

\begin{lemma}\label{l:transfer-lemma}
Let $\cZ\subseteq\cX$ be a distal extension with $\cX$ ergodic and let $\cY$ be a separable factor of $\cX$. Then there are separable factors $\cZ' \subseteq \cY' \subseteq \cX$ with $\cZ' \subseteq \cZ$, $\cY \subseteq \cY'$ and such that the extension $\cZ' \subseteq \cY'$ is distal; moreover, $\rk(\cY'/\cZ')\leq \rk(\cX/\cZ)$.
\end{lemma}
\begin{proof}
We proceed by induction on the distal rank $\alpha=\rk(\cX/\cZ)$. Let $(\cX_\beta)_{\beta\leq\alpha}$ be a distal tower for the extension $\cZ\subseteq\cX$. In the base case $\alpha=0$, we take $\cZ'=\cY'=\cY$.

Suppose next that $\alpha=\beta+1$. There are finitely generated $\Gamma$-invariant $\cX_\beta$-modules $M_i$ such that $L^2(\cX_\alpha)=\bigoplus_{i\in I} M_i$ and we can choose a countable subset $J\subseteq I$ such that $\bigoplus_{i\in J}M_i$ contains $L^2(\cY)$. Fix some $M_i$ and a corresponding $\cX_\beta$-basis $e_1,\dots,e_n$. Since $\Gamma$ is countable and $\cY$ is separable, we can find a separable factor $\cW \subseteq \cX_\beta$ such that the $\cW$-module $M'_i$ generated by $e_1,\dots,e_n$ is $\Gamma$-invariant and contains the intersection $L^2(\cY)\cap M_i$; moreover, since $J$ is countable, we can choose $\cW$ which works for every $i\in J$. Then we consider the sum $M=\bigoplus_{i\in J} M'_i$. By adding all the product modules (as defined in the proof of Proposition~\ref{p:factor-generated-by-compact-factors}) if necessary, we may assume that $M$ defines a separable factor $\cY_0\subseteq\cX$ (i.e., $M=L^2(\cY_0)$) which contains $\cY$ and is a compact extension of $\cW$. Now we can apply the inductive hypothesis to get a distal extension $\cZ'\subseteq \cY_1$ of separable factors of $\cX_\beta$ such that $\cZ'\subseteq \cZ$ and $\cW \subseteq \cY_1$, and moreover $\rk(\cY_1/\cZ')\leq \beta$. Let $\cY'$ be the factor generated by $\cY_0$ and $\cY_1$. Then $\cY'$ is separable, contains $\cY$, and is a compact extension of $\cY_1$, hence also a distal extension of $\cZ'$ with $\rk(\cY'/\cZ')\leq \alpha$. The pair $\cZ',\cY'$ is as required.

Finally, if $\alpha$ is a limit ordinal, we can choose countably many $\beta_n<\alpha$ such that $\cY$ is generated by the factors $\cY\cap\cX_{\beta_n}$. By the inductive hypothesis, we have separable distal extensions $\cZ_n \subseteq\cY_n$, of rank less than $\alpha$, with $\cZ_n\subseteq \cZ$ and $\cY\cap\cX_{\beta_n} \subseteq \cY_n$. If $\cZ'$ and $\cY'$ are the factors generated respectively by the $\cZ_n$ and the $\cY_n$, then they are separable, $\cZ'\subseteq\cZ$, $\cY\subseteq\cY'$, and the extension $\cZ'\subseteq\cY'$ is distal of rank at most $\alpha$ by Proposition~\ref{p:factor-generated-by-distal-factors}.
\end{proof}

The following proposition states that factors of distal systems are distal; see also Proposition~\ref{p:rank-of-factors} for a refinement.
\begin{prop}\label{p:factor-of-distal-is-distal}
Suppose $\cZ\subseteq \cY\subseteq \cX$ and $\cX$ is ergodic. Then, the extension $\cZ\subseteq \cX$ is distal if and only if the extensions $\cZ\subseteq \cY$ and $\cY\subseteq\cX$ are distal.
\end{prop}
\begin{proof}
The only difficulty is in proving that if $\cZ\subseteq \cX$ is distal, then so is $\cZ\subseteq \cY$. So assume the former extension is distal.

If $\cX_1$ and $\cX_2$ are factors of $\cX$, we will denote by $\cX_1\cX_2$ the factor generated by $\cX_1$ and $\cX_2$. We prove first that if $\cW\subseteq\cX$ is a separable factor, then $\cZ\cW$ is a distal extension of $\cZ$. By Lemma~\ref{l:transfer-lemma}, there are separable factors $\cZ' \subseteq \cW' \subseteq \cX$ with $\cZ' \subseteq \cZ$, $\cW \subseteq \cW'$ such that the extension $\cZ' \subseteq \cW'$ is distal. Then, by \cite{Glasner2003}*{Theorem~10.38}, $\cZ'\cW$ is also a distal extension of $\cZ'$. Now applying Proposition~\ref{p:factor-generated-by-distal-factors} to the two distal extensions $\cZ'\cW \supseteq \cZ'$ and $\cZ \supseteq \cZ$ yields that the extension $\cZ\cW \supseteq \cZ$ is distal.

Hence, since for each separable factor $\cW$ of $\cY$ the extension $\cZ\subseteq \cZ\cW$ is distal, the extension $\cZ\subseteq \cY$ is distal as well, again by Proposition~\ref{p:factor-generated-by-distal-factors}.
\end{proof}

We end this section with a well-known amalgamation result which is usually stated and proved for separable systems in the literature.
\begin{prop}
  \label{p:ergodic-joinings}
  Let $i_1 \colon \cZ \to \cX_1$ and $i_2 \colon \cZ \to \cX_2$ be embeddings of ergodic systems. Then there exists an ergodic system $\cY$ and embeddings $j_1 \colon \cX_1 \to \cY$ and $j_2 \colon \cX_2 \to \cY$ such that $j_1 \circ i_1 = j_2 \circ i_2$. If $\cY$ is moreover generated by $j_1(\cX_1)$ and $j_2(\cX_2)$, we say it is an \df{ergodic joining of $\cX_1$ and $\cX_2$ over $\cZ$}.
\end{prop}
\begin{proof}
  For simplicity of notation, we will identify $\cZ$ with its images $i_1(\cZ) \sub \cX_1$ and $i_2(\cZ) \sub \cX_2$. Let $X_1$ and $X_2$ be the Stone spaces of the Boolean algebras $\cX_1$ and $\cX_2$ and let $\mu_1$ and $\mu_2$ be the corresponding measures on $X_1$ and $X_2$. We identify $\cX_i$ with the algebra of clopen sets of $X_i$. Denote by $\pi_i \colon X_1 \times X_2 \to X_i$ the projection maps.
  
  Let $P_\Gamma(X_1 \times X_2)$ denote the compact, convex set of Radon probability measures on $X_1 \times X_2$ that are invariant under the diagonal action of $\Gamma$. We consider
  \begin{multline*}
    J_\cZ(\cX_1, \cX_2) \coloneqq \set[\big]{\nu \in P_\Gamma(X_1 \times X_2) : (\pi_1)_* \nu = \mu_1, (\pi_2)_*\nu = \mu_2 \And \\
      \nu\big( (a \times X_2) \sdiff (X_1 \times a) \big) = 0 \text{ for all } a \in \cZ},
  \end{multline*}
  the set of \emph{joinings of $\cX_1$ and $\cX_2$ over $\cZ$}, which is still compact and convex. Next we check that it is non-empty.
  For $a_1\in \cX_1$, $a_2\in\cX_2$, we define
  \begin{equation*}
  \lambda(a_1 \times a_2) \coloneqq \int_Z \mu_1(a_1|\cZ) \mu_2(a_2|\cZ) \ud z.
\end{equation*}
It is easy to see that $\lambda$ extends to a finitely additive $\Gamma$-invariant measure on the Boolean algebra generated by the clopen rectangles. Thus $\lambda$ defines a $\Gamma$-invariant linear functional on the dense subspace of $C(X_1 \times X_2)$ consisting of step functions over clopen rectangles. As it is also continuous for the $\sup$ norm, it extends to a linear functional on all of $C(X_1 \times X_2)$, i.e., to a Radon measure on $X_1\times X_2$. One readily checks that this measure belongs to $J_{\cZ}(\cX_1, \cX_2)$.

  By the Krein--Milman theorem, there exists an extreme point $\lambda_0$ of $J_\cZ(\cX_1, \cX_2)$. We check that the system $\Gamma\actson (X_1 \times X_2, \cB, \lambda_0)$, where $\cB$ is the Borel $\sigma$-algebra, is ergodic. For this, it is enough to see that $\lambda_0$ is an extreme point of $P_\Gamma(X_1 \times X_2)$.
Suppose that $\lambda_0 = t \lambda_1 + (1-t) \lambda_2$ with $\lambda_1, \lambda_2 \in P_\Gamma(X_1 \times X_2)$ and $0 < t < 1$. Then $\mu_1 = (\pi_1)_*\lambda_0 = t (\pi_1)_*\lambda_1 + (1-t) (\pi_1)_*\lambda_2$, which by the ergodicity of $\mu_1$ implies that $(\pi_1)_*\lambda_1 = (\pi_1)_* \lambda_2 = \mu_1$; similarly, $(\pi_2)_*\lambda_1 = (\pi_2)_* \lambda_2 = \mu_2$. The last condition in the definition of $J_\cZ(\cX_1, \cX_2)$ is obviously satisfied because of positivity.
\end{proof}

The pmp system induced by the action of $\Gamma$ on $(X_1\times X_2,\cB,\lambda)$, where $\lambda$ is as in the proof of the previous proposition, is called the \df{relatively independent joining of $\cX_1$ and $\cX_2$ over $\cZ$}, and is denoted by $\cX_1\otimes_\cZ\cX_2$.

%%%%%%%%%%%%%%%%%%%%%%%%%%%%%%%%%%%%%%%%%%%%%%%%%%

\section{Universal ergodic distal systems}
\label{sec:universal-distal-systems}

Next we construct the largest ergodic compact extension and the largest ergodic distal extension of an ergodic system. We start by noting that we can amalgamate ergodic compact extensions.

\begin{lemma}
  \label{l:compact-ext-prop}
  Let $\cZ$ be an ergodic system. Then the family of compact, ergodic extensions of $\cZ$ has the following properties:
  \begin{enumerate}
  \item \label{i:l:compact-ext-prop:1} It is directed: if $i_1 \colon \cZ \to \cX_1$ and $i_2 \colon \cZ \to \cX_2$ are compact, ergodic extensions of $\cZ$, then there exists a compact, ergodic extension $j \colon \cZ \to \cX_3$ and embeddings $j_1 \colon \cX_1 \to \cX_3$ and $j_2 \colon \cX_2 \to \cX_3$ such that $j_1 \circ i_1 = j_2 \circ i_2$.
    
  \item \label{i:l:compact-ext-prop:2} It is closed under direct limits.
  \end{enumerate}
\end{lemma}
\begin{proof}
  \ref{i:l:compact-ext-prop:1} This follows from Propositions \ref{p:ergodic-joinings} and \ref{p:factor-generated-by-compact-factors}.

  \ref{i:l:compact-ext-prop:2} Let $\cX=\varinjlim \cX_\alpha$ be a direct limit of ergodic compact extensions of~$\cZ$, and let us identify the factors $\cX_\alpha$ with subsets of $\cX$. Then $\cX$ is generated by the $\cX_\alpha$, hence is also a compact extension of $\cZ$. Ergodicity of $\cX$ follows readily from the consideration of the zero-mean invariant vectors in $L^2(\cX)$: their projection on each $L^2(\cX_\alpha)$ is invariant, thus equal to zero; since the spaces $L^2(\cX_\alpha)$ generate $L^2(\cX)$, this is enough.
\end{proof}

Next we show how to amalgamate the set of \emph{all} ergodic compact extensions of a given system $\cZ$. If $\cZ$ is the trivial one-point system, this construction yields the well-known Bohr compactification of the group. We are grateful to Ehud Hrushovski for supplying the simple uniqueness argument below, which is more general than the one we originally had.

A system $\cX$ is said to be \emph{coalescent over a factor $\cZ$} if every $\cZ$-embedding $\cX \to \cX$ (i.e., a self-embedding fixing $\cZ$ pointwise) is an isomorphism.

\begin{theorem}
  \label{th:absolute-max-compact}
  Let $\cZ$ be any ergodic $\Gamma$-system. Then there exists a unique compact, ergodic extension $i_0 \colon \cZ \to \widehat{\cZ}$ with the following property: if $i \colon \cZ \to \cX$ is any compact, ergodic extension, then there exists an embedding $j \colon \cX \to \widehat{\cZ}$ such that $i_0 = j \circ i$. Moreover $|\widehat \cZ| \leq (|\cZ| + \aleph_0)^{\aleph_0}$ and $\widehat \cZ$ is coalescent over $\cZ$.
\end{theorem}
\begin{proof}
  First we show existence. By Proposition~\ref{p:comp-ext-size-bound}, there exists a set $\set{\cX_\alpha : \alpha < \kappa}$ of representatives of all $\cZ$-isomorphism classes of compact, ergodic extensions of $\cZ$. We use Lemma~\ref{l:compact-ext-prop} to define another sequence $(\cX'_\alpha)_{\alpha \leq \kappa}$ of ergodic compact extensions of $\cZ$ with the property that $\cX_\alpha \sub \cX'_{\alpha}$ and $\cX'_\alpha \subseteq \cX'_{\alpha'}$ for $\alpha < \alpha' \leq \kappa$. We set $\cX'_0=\cX_0$. For $\cX'_{\alpha+1}$ we take any amalgam of $\cX_\alpha$ and $\cX'_\alpha$ over $\cZ$. Finally, if $\lambda \leq \kappa$ is a limit ordinal, we let $\cX'_\lambda$ be the direct limit of the $\cX'_\alpha$ for $\alpha < \lambda$. The last system $\widehat \cZ \coloneqq \cX'_\kappa$ is an ergodic compact extension of $\cZ$ that contains any other such extension as a factor. The bound on the density of $\widehat \cZ$ follows from Proposition~\ref{p:comp-ext-size-bound}.

Uniqueness follows from coalescence. To prove coalescence, suppose that $\phi \colon \widehat{\cZ} \to \widehat{\cZ}$ is a $\cZ$-embedding which is not surjective. Let $\kappa = |\widehat{\cZ}|$ and consider the directed system $(\cZ_\alpha : \alpha < \kappa^+)$ where each $\cZ_\alpha$ is equal to $\widehat{\cZ}$, each map $\cZ_\alpha \to \cZ_{\alpha+1}$ is equal to $\phi$, and for limit $\lambda$ one takes the direct limit $\cZ'_\lambda=\varinjlim_{\alpha < \lambda} \cZ_\alpha$ and a $\cZ$-embedding $\cZ'_\lambda\to\widehat{\cZ}=\cZ_\lambda$, and then defines the maps $\cZ_\alpha\to\cZ_\lambda$ for $\alpha<\lambda$ by the composition $\cZ_\alpha\to\cZ'_\lambda\to\cZ_\lambda$. Then the limit $\varinjlim_{\alpha<\kappa^+} \cZ_\alpha$  is an ergodic, compact extension of $\cZ$ of density $\kappa^+$, which contradicts the maximality of $\widehat{\cZ}$.
\end{proof}

The uniqueness of the construction has the following direct consequence.

\begin{cor}\label{c:autZ-to-authatZ}
Every automorphism of $\cZ$ extends to an automorphism of $\widehat{\cZ}$.
\end{cor}
\begin{proof}
If $i_0 \colon \cZ \to \widehat \cZ$ is as in the theorem and $\phi \colon \cZ \to \cZ$ is an automorphism, then $i = i_0 \circ \phi\colon \cZ \to \widehat \cZ$ is also a universal, compact, ergodic extension of $\cZ$. By uniqueness, there is an isomorphism $\psi \colon \widehat \cZ \to \widehat \cZ$ such that $\psi \circ i_0 = i = i_0 \circ \phi$.
\end{proof}

 We can iterate the construction above and define an ergodic distal extension $\cZ\subseteq\cD_\alpha(\cZ)$ for each ordinal number $\alpha$ in the natural way:
\begin{itemize}
\item $\cD_0(\cZ)=\cZ$;
\item $\cD_{\alpha+1}(\cZ)=\widehat{\cD_\alpha(\cZ)}$;
\item $\cD_\lambda(\cZ)=\varinjlim_{\beta<\lambda} \cD_\beta(\cZ)$ if $\lambda$ is a limit ordinal.
\end{itemize}
Note that the tower $\cD_\alpha(\cZ)$ is unique up to $\cZ$-isomorphism. More precisely, if $\alpha$ is an ordinal and $(\cD_\beta)_{\beta\leq\alpha}$, $(\cD'_\beta)_{\beta\leq\alpha}$ are two towers constructed as above, then there is an isomorphism $\cD_\alpha\to\cD'_\alpha$ that restricts to a $\cZ$-isomorphism $\cD_\beta\to\cD'_\beta$ for each $\beta<\alpha$. This follows from Theorem~\ref{th:absolute-max-compact} and an easy induction using Corollary~\ref{c:autZ-to-authatZ}. Clearly, each $\cD_\alpha(\cZ)$ is a distal extension of $\cZ$ of distal rank at most $\alpha$.

\begin{theorem}\label{th:universal-distal-systems}
Let $\cZ$ be an ergodic system. Then the system $\cD_\alpha(\cZ)$ is universal for the ergodic, distal extensions of $\cZ$ of rank at most $\alpha$, i.e., any such system is a factor of $\cD_\alpha(\cZ)$. Moreover, $\cD_\alpha(\cZ)$ is coalescent over $\cZ$ and hence unique with this universal property.
\end{theorem}
\begin{proof}
  Let $(\cX_\beta)_{\beta\leq\alpha}$ be a distal tower over $\cZ$ of length $\alpha$ with $\cX_\alpha$ ergodic. We construct by induction a sequence of embeddings $i_\beta \colon \cX_\beta \to \cD_\beta(\cZ)$. We start with $i_0 = \id$. For the successor step, suppose that $i_\beta$ has already been constructed. Apply Proposition~\ref{p:ergodic-joinings} to find an ergodic joining $\cY$ of $\cX_\beta \to \cX_{\beta+1}$ and $\cX_\beta \xrightarrow{i_\beta} \cD_\beta(\cZ)$ over $\cX_\beta$. By Proposition~\ref{p:factor-generated-by-compact-factors}, $\cY$ is a compact extension of $\cD_\beta(\cZ)$, so by the universality of $\cD_{\beta+1}$, there exists an embedding $\cY \to \cD_{\beta+1}(\cZ)$ over $\cD_\beta(\cZ)$. Now composing the embeddings $\cX_{\beta +1} \to \cY$ and $\cY \to \cD_{\beta+1}(\cZ)$ yields $i_{\beta+1}$. If $\lambda$ is a limit ordinal, one can just take $i_\lambda = \bigcup_{\alpha < \lambda} i_\beta$.
  
For the moreover assertion, we want to show that each self-embedding of $\cD_\alpha(\cZ)$ fixing $\cZ$ is an isomorphism. We proceed by induction. Suppose this holds for every $\alpha<\alpha'$ and let $\phi\colon \cD_{\alpha'}(\cZ)\to \cD_{\alpha'}(\cZ)$ be a $\cZ$-embedding. We show first that $\phi(\cD_\alpha(\cZ))\subseteq \cD_\alpha(\cZ)$ for each $\alpha<\alpha'$. Indeed, by Proposition~\ref{p:factor-generated-by-distal-factors}, the system $\cW$ generated by $\phi(\cD_\alpha(\cZ))$ and $\cD_\alpha(\cZ)$ inside $\cD_{\alpha'}(\cZ)$ is a distal extension of $\cZ$ of rank at most $\alpha$, hence a factor of $\cD_\alpha(\cZ)$ by the universal property we just proved. Thus we have embeddings $\cD_\alpha(\cZ)\to\cW\to\cD_\alpha(\cZ)$, and the inductive hypothesis implies they are surjective; hence $\phi(\cD_\alpha(\cZ))\subseteq \cD_\alpha(\cZ)$. Again by the inductive hypothesis, we actually have equality for each $\alpha<\alpha'$. Now if $\alpha'$ is of the form $\alpha+1$, the fact that $\phi$ is surjective follows from Theorem~\ref{th:absolute-max-compact}. The limit case is clear.
\end{proof}

The proof of Theorem~\ref{th:universal-distal-systems} implies that if $\alpha < \alpha'$ and $\cY$ is a distal extension of $\cZ$ of rank at most $\alpha$ such that $\cY \sub \cD_{\alpha'}(\cZ)$, then $\cY \sub \cD_\alpha(\cZ)$. Thus within $\cD_{\alpha'}(\cZ)$, the system $\cD_\alpha(\cZ)$ can be defined as the maximal intermediate distal extension of $\cZ$ of rank at most $\alpha$.

\begin{theorem}\label{th:distal-god}
Let $\cZ$ be an ergodic system and let $\kappa = |\cZ| + \aleph_0$. Then we have $\cD_{\kappa^++1}(\cZ)=\cD_{\kappa^+}(\cZ)$, i.e., $\cD_{\kappa^+}(\cZ)$ is the largest ergodic, distal extension of $\cZ$. In particular, every ergodic, distal extension of $\cZ$ has density at most $\kappa^{\aleph_0} + \kappa^+$.
\end{theorem}
\begin{proof}
The bound on the rank follows from Corollary~\ref{c:bound-on-distal-rank}. Then, the density bound follows from Theorem~\ref{th:absolute-max-compact}.
\end{proof}

In particular, if we define $\cD_\alpha(\Gamma)\coloneqq\cD_\alpha(1)$ where $1$ is the trivial $1$-point system, then $\cD_{\omega_1}(\Gamma)$ is the largest ergodic, distal $\Gamma$-system. It is natural to ask what the actual distal rank of $\cD_{\omega_1}(\Gamma)$ is. For $\Gamma=\Z$ this is answered by a construction of Beleznay and Foreman.

\begin{theorem}\label{th:rank-of-Domega1Z}
We have $\cD_\alpha(\Z)\subsetneq \cD_{\alpha+1}(\Z)$ for all $\alpha < \omega_1$. That is, $\rk(\cD_{\omega_1}(\Z))=\omega_1$.
\end{theorem}
\begin{proof}
In \cite{Beleznay1995} examples are built of ergodic, distal systems of arbitrarily high countable rank. Since the rank of a distal system is at least as large as the rank of any of its factors (see the Appendix), the result follows.
\end{proof}

On the other hand, there are groups $\Gamma$ with no non-trivial compact systems (for example, finitely generated, simple, infinite groups) for which $\cD_{\omega_1}(\Gamma)=1$. We will show in Section~\ref{sec:groups-with-property-T} that for groups with property (T), all ergodic distal systems are compact, i.e.,  $\cD_{\omega_1}(\Gamma) = \cD_1(\Gamma)$.

The systems $\cD_\alpha(\Gamma)$ give examples where many classical results in ergodic theory valid in the separable setting fail. We will mention two.

Zimmer~\cite{Zimmer1977} proved that if $(X, \cX, \mu)$ is any separable $\Z$-system and $K$ is a non-trivial, compact, metrizable group, then there exists a cocycle $\rho \colon \Z \times X \to K$ which is not a coboundary. However, we have the following.
\begin{prop}
  \label{p:Zimmer-fail}
  Let $\cX = \cD_{\omega_1}(\Gamma)$ and let $X$ be the probability space associated to $\cX$. Suppose that $K$ is a compact group and $\rho \colon \Gamma \times X \to K$ is a cocycle. Then $\rho$ is a coboundary.
\end{prop}
\begin{proof}
  Let $L \leq K$ be the \df{essential range} of $\rho$. Then $X \times_\rho L$ is an ergodic, compact extension of $\cD_{\omega_1}(\Gamma)$, so it must be trivial. This means that $L$ is trivial and $\rho$ is a coboundary.
\end{proof}

Recall that an ergodic system $\cX$ is called \df{$2$-simple} if whenever $\cX_1$ and $\cX_2$ are two isomorphic copies of $\cX$ inside a larger ergodic system, then $\cX_1 = \cX_2$. If $\cX$ is not weakly mixing (for example, distal), $2$-simplicity implies simplicity (which has a more complicated definition; see \cite{Glasner2003}*{Chapter~12} for more details).
\begin{prop}
  \label{p:Dalpha-simple}
  For every ordinal $\alpha \leq \omega_1$, $\cD_\alpha(\Gamma)$ is a simple system.
\end{prop}
\begin{proof}
  Let $\cX = \cD_\alpha(\Gamma)$ and suppose that $\cX_1$ and $\cX_2$ are two copies of $\cX$ in an ergodic system $\cY$. Then by Proposition~\ref{p:factor-generated-by-distal-factors}, the system generated by $\cX_1$ and $\cX_2$ is distal of rank at most $\alpha$. Now by universality and coalescence of $\cX$, we must have that $\cX_1 = \cX_2$.
\end{proof}

Proposition~\ref{p:Dalpha-simple} should be contrasted with a theorem of Veech (see \cite{Glasner2003}*{Theorem~12.1}) that states that a separable, ergodic, simple system is either compact or weakly mixing. Combining the two, we have the following (see also Theorem~\ref{th:polish-AutD2} for a generalization).
\begin{cor}
  \label{c:from-Veech}
  Suppose that $\cD_2(\Gamma)$ is a separable system. Then $\cD_2(\Gamma) = \cD_1(\Gamma)$ and every distal $\Gamma$-system is compact.
\end{cor}

We conclude this section with a simple observation for future reference.
\begin{prop}
  \label{p:inclusion-Dalpha}
  Let $\cX \sub \cY$ be an extension of ergodic systems. Then for every ordinal~$\alpha$, $\cD_\alpha(\cX)$ is a factor of $\cD_\alpha(\cY)$.
\end{prop}
\begin{proof}
  Let $\cW$ be an ergodic joining of $\cY$ and $\cD_\alpha(\cX)$ over $\cX$, which exists by Proposition~\ref{p:ergodic-joinings}. Then by Proposition~\ref{p:factor-generated-by-distal-factors}, the extension $\cY \sub \cW$ is distal of rank at most $\alpha$ and by the universal property of $\cD_\alpha(\cY)$, $\cW$ is a factor of $\cD_\alpha(\cY)$.
\end{proof}

%%%%%%%%%%%%%%%%%%%%%%%%%%%%%%%%%%%%%%%%%%%%%%%%%%

\section{Compact extensions and algebraic closure}
\label{sec:comp-extens-algebr}

Recall that an element of a structure $M$ is \emph{algebraic over a subset $A\subseteq M$} if it belongs to a compact $A$-definable subset of $M$. We say $a\in M$ is \emph{algebraic} if it is algebraic over $\emptyset$. The set $\acl^M(A)$ of all elements algebraic over $A$ is the \emph{algebraic closure of $A$ in $M$}, and forms a substructure of $M$.

The substructure $\acl^M(\emptyset)\subseteq M$ is a fundamental model-theoretic invariant of $M$: it only depends on its first-order theory and it is a common substructure of all models of this theory. In the case of a $\Gamma$-system~$\cX$, this defines a canonical factor
$$\cX^{\alg} \coloneqq \acl^\cX(\emptyset) \subseteq\cX$$
and one may ask whether this is related to any natural notion in ergodic theory. If $\Gamma$ is an infinite amenable group and the action $\Gamma \actson \cX$ is free, then $\acl^\cX(A)$ is just the factor generated by $A$ in $\cX$ (this follows from quantifier elimination). Hence $\cX^{\alg}$ is always the trivial factor. On the other hand, as we show below, for strongly ergodic actions the situation is completely different: if $\cX$ is strongly ergodic, then $\cX^{\alg}$ contains all distal factors of $\cX$. We will actually prove a stronger result by showing that algebraicity is witnessed by quantifier-free formulas (cf.\ Theorem~\ref{th:main-lemma}).

Given a system $\cX$, a factor $\cZ$, and a tuple $F=(f_i)_{i\in I} \subset L^\infty(\cX)$, the \df{quantifier-free type of $(f_i)$ over $\cZ$} is given by the joint conditional distribution over $\cZ$ of the tuple of random variables $(\gamma \cdot f_i)_{\gamma \in \Gamma, i \in I}$. If $F$ and $F'$ have the same quantifier-free type over $\cZ$, we will write $F \equiv_\cZ F'$. In other words, if $F = (f_i)$ and $F' = (f'_i)$, we have that $F \equiv_\cZ F'$ iff for every $*$-polynomial $P(x_0, \ldots, x_{n-1})$ with coefficients from $L^\infty(\cZ)$, indices $i_0, \ldots, i_{n-1} \in I$ and $\gamma_0, \ldots, \gamma_{n-1} \in \Gamma$,
\begin{equation*}
  \E \big(P(\gamma_0 \cdot f_{i_0}, \ldots, \gamma_{n-1} \cdot f_{i_{n-1}})\big) = \E \big(P(\gamma_0 \cdot f_{i_0}', \ldots, \gamma_{n-1} \cdot f_{i_{n-1}}')\big).
\end{equation*}
In particular, the system generated by $\cZ$ and $F$ is isomorphic to the one generated by $\cZ$ and $F'$.

The following is one of our main results.

\begin{theorem}\label{th:main-lemma}
Suppose $\cZ\subseteq\cX$ is a distal extension with $\cX$ strongly ergodic and $f_0 \in L^\infty(\cX)$. Then the set $\set{f \in L^\infty(\cX) : f \equiv_\cZ f_0}$ is compact in the $L^2$-norm.
\end{theorem}
\begin{proof}
  It is clear that the set is closed, so it is enough to prove that it is totally bounded. Suppose for contradiction that there exist $\eps > 0$ and $\set{f_i : i < \omega} \subseteq L^\infty(\cX)$ such that $f_i \equiv_\cZ f_0$ for all $i$ and $\nm{f_i - f_j}_2 \geq \eps$ for all $i \neq j$. Since $\cX$ is a distal extension of $\cZ$, by Proposition~\ref{p:factor-of-distal-is-distal}, the factor generated by $\cZ$ and $f_0$ is a distal extension of $\cZ$.

In view of Theorem~\ref{th:distal-god}, let $\kappa$ be a cardinal larger than the density of any ergodic, distal extension of $\cZ$. It follows from the compactness theorem that there exists an elementary extension $\cX'$ of $\cX$ and a sequence $(f'_i)_{i<\kappa}\subseteq L^\infty(\cX')$ such that $f_i'\equiv_\cZ f_0$ for all $i<\kappa$ and $\nm{f_i' - f_j'} \geq \eps$ for $i \neq j$. The first condition ensures that for every $i < \kappa$, the factor generated by $f'_i$ over $\cZ$ is a distal extension of $\cZ$. Hence, by Proposition~\ref{p:factor-generated-by-distal-factors}, the factor $\cZ'$ generated by all the $f'_i$ over $\cZ$ is also a distal extension of $\cZ$ and, moreover, $|\cZ'|\geq\kappa$ since $(f'_i)_{i<\kappa}$ is separated. On the other hand, $\cX'$ is ergodic since $\cX$ is strongly ergodic, and hence so is $\cZ'\subseteq \cX'$. This contradicts the choice of $\kappa$.
\end{proof}
\begin{remark}
  \label{rem:after-main-lemma}
  If in Theorem~\ref{th:main-lemma} one assumes that the extension $\cZ \sub \cX$ is compact rather than distal, one does not need the strong ergodicity assumption on $\cX$ and the result follows easily from Lemma~\ref{l:dim-at-most-n}. Thus a more direct approach to Theorem~\ref{th:distal-strongly-erg-acl} below and many of the applications would use Lemma~\ref{l:dim-at-most-n} and then rely on the transitivity of the algebraic closure operator.
\end{remark}

Theorem~\ref{th:main-lemma} has the following model-theoretic consequence.
\begin{theorem}
  \label{th:distal-strongly-erg-acl}
  Let $\cX$ be a strongly ergodic system and $\cZ \sub \cX$ be a factor. Then any distal extension of $\cZ$ in $\cX$ is contained in $\acl^\cX(\cZ)$.
\end{theorem}
\begin{proof}
  If $f, g \in L^\infty(\cX)$, we will write $\tp^\cX(f/\cZ) = \tp^\cX(g/\cZ)$ if $f$ and $g$ have the same type over $\cZ$, i.e., for any first-order formula $\phi(v)$ with parameters from $\cZ$, we have $\phi^\cX(f) = \phi^\cX(g)$. Let $\cX'$ be a sufficiently saturated elementary extension of $\cX$.
  Recall that for $f \in L^\infty(\cX)$,
  \begin{equation*}
    f \in \acl^\cX(\cZ) \iff \set{g \in L^\infty(\cX') : \tp^{\cX'}(g/\cZ) = \tp^{\cX'}(f/\cZ)} \text{ is compact}.
  \end{equation*}
  
  Let now $\cZ \sub \cY \sub \cX$ with $\cZ \sub \cY$ distal and let $f \in L^\infty(\cY)$. Note that
  \begin{equation}
    \label{eq:1}
    \set{g \in L^\infty(\cX') : \tp^{\cX'}(g/\cZ) = \tp^{\cX'}(f/\cZ)} \sub \set{g \in L^\infty(\cX') : g \equiv_\cZ f}.
  \end{equation}
  Call the set on the right-hand side of \eqref{eq:1} $A$ and let $\cW$ be the factor of $\cX'$ generated by $A$ and $\cZ$. By Propositions~\ref{p:factor-of-distal-is-distal} and \ref{p:factor-generated-by-distal-factors}, the extension $\cZ \sub \cW$ is distal. Furthermore, $\cW$ is strongly ergodic as a factor of the strongly ergodic system $\cX'$. So by Theorem~\ref{th:main-lemma}, $A$ is compact and thus $f \in \acl^{\cX}(\cZ)$.
\end{proof}

These results have a number of corollaries.
\begin{cor}
  \label{c:Aut(distal)-is-compact}
  Let $\cZ\subseteq \cX$ be a distal extension with $\cX$ strongly ergodic. Then $\Aut_\cZ(\cX)$ is compact.
\end{cor}

\begin{cor}
  \label{c:distal-is-in-Xalg}
  Let $\cX$ be a strongly ergodic system and let $\cD$ be its maximal distal factor. Then $\cD \sub \cX^{\alg}$. In particular, $\cD$ is separable.
\end{cor}
\begin{proof}
The algebraic closure of the empty set is always separable if the language is countable.
\end{proof}
\begin{remark}
  \label{rem:Bohr-T}
  Corollary~\ref{c:distal-is-in-Xalg} implies that any distal, ergodic system of a property~(T) group is separable. In particular, the Bohr compactification of a property (T) group $\Gamma$ is second countable (or, which is equivalent, $\Gamma$ admits only countably many finite-dimensional unitary representations). This is a result of Wang~\cite{Wang1975}.
\end{remark}

\begin{cor}
  \label{c:Robin}
  Suppose $\cX$ is distal, strongly ergodic and the action $\Aut(\cX)\actson \cX$ is ergodic. Then $\cX$ is compact.
\end{cor}
\begin{proof}
By Corollary~\ref{c:Aut(distal)-is-compact}, $K=\Aut(\cX)$ is compact. Let $L^2(\cX)=\bigoplus_i M_i$ be a decomposition into irreducible representations of $K$. For each $i$, let $E_i$ denote the subspace generated by all the subrepresentations of $K \actson L^2(\cX)$ isomorphic to $M_i$. The ergodicity assumption and Lemma~\ref{l:dim-at-most-n} (which works for actions of arbitrary groups) imply that the $E_i$ are finite-dimensional. They are also  $\Gamma$-invariant and generate $L^2(\cX)$. This shows that $\cX$ is a compact system, as required.
\end{proof}

In \cite{Gaboriau2016p}, Gaboriau, Ioana and Tucker-Drob prove a cocycle superrigidity result for compact actions of product groups and countable targets. Corollary~\ref{c:Robin} above implies the following parallel result. We are grateful to Robin Tucker-Drob for suggesting this corollary.
\begin{cor}
  \label{c:Robin2}
  Let $\Gamma$ and $\Lambda$ be countable groups and $\Gamma \times \Lambda \actson X$ be a distal action such that the action $\Gamma \actson X$ is strongly ergodic and the action $\Lambda \actson X$ is ergodic. Then the original action $\Gamma \times \Lambda \actson X$ is compact.
\end{cor}

The following is the main result of \cite{Ioana2016} (for separable systems).
\begin{cor}[Ioana--Tucker-Drob]
  \label{c:full-Ioana-TD}
  Let $\cX$ be a distal system and $\cY$ be strongly ergodic. If $\cX <_w \cY$, then $\cX$ is a factor of $\cY$.
\end{cor}
\begin{proof}
  By Lemma~\ref{l:wcon-existential}, the condition $\cX <_w \cY$ says that $\cX$ embeds as a factor of some elementary extension $\cY'$ of $\cY$. Thus, by Theorem~\ref{th:distal-strongly-erg-acl}, $\cX \sub \cY'^{\alg}$. As the algebraic closure does not change with elementary extensions, this implies that $\cX \sub \cY^{\alg} \sub \cY$.
\end{proof}

The previous corollaries all follow from Theorem~\ref{th:distal-strongly-erg-acl}. Next we have an application that uses the full power of the Theorem~\ref{th:main-lemma}. Note that every ergodic, compact system is coalescent but this fails in general for ergodic, distal systems; see \cite{Parry1970} for an example of a non-coalescent ergodic $\Z$-action of distal rank $2$.

\begin{theorem}
  \label{th:distal-coalescent}
  Every strongly ergodic, distal system is coalescent. More generally, if $\cZ\subseteq \cX$ is a distal extension with $\cX$ strongly ergodic, then $\cX$ is coalescent over $\cZ$.
\end{theorem}
\begin{proof}
Suppose $\phi\colon L^\infty(\cX)\to L^\infty(\cX)$ is a $\Gamma$-embedding that is the identity on $L^\infty(\cZ)$, but $f\in L^\infty(\cX)$ is not in the image of $\phi$. Let $\epsilon=d(f,\phi(L^\infty(\cX)))>0$. Then for all $i < j < \omega$, $\phi^i(f) \equiv_\cZ \phi^j(f)$ and $d(\phi^i(f),\phi^j(f)) \geq \epsilon$. This contradicts Theorem~\ref{th:main-lemma}.
\end{proof}
\begin{remark}
  \label{rem:Ioana}
  After reading a preliminary version of this paper, Adrian Ioana pointed out to us that it is possible to use \cite{Ioana2016}*{Lemma~2.6} (which implies that the semigroup of endomorphisms of a strongly ergodic, distal system is compact) instead of Theorem~\ref{th:main-lemma} in the proof above.
\end{remark}

\begin{cor}
  \label{c:distal-wequiv}
  Let $\cX$ and $\cY$ be two strongly ergodic, distal systems which are weakly equivalent. Then they are isomorphic.
\end{cor}
\begin{proof}
  By Corollary~\ref{c:full-Ioana-TD}, $\cX$ and $\cY$ are bi-embeddable, which combined with Corollary~\ref{th:distal-coalescent} yields that they are isomorphic.
\end{proof}

%%%%%%%%%%%%%%%%%%%%%%%%%%%%%%%%%%%%%%%%%%%%%%%%%%

\section{Proof by model-theoretic stability}
\label{sec:two-line-proof}

In this section, we give a simple proof of Theorem~\ref{th:distal-strongly-erg-acl} that only uses basic stability theory. It requires more familiarity with model-theoretic notions than the previous sections.

We will make use of the \emph{stable independence relation}, i.e., the ternary relation $\Findep[T]{}$ that can be defined for any (possibly unstable) complete theory $T$ and that corresponds to non-forking restricted to stable formulas. We refer the reader to \cite{BenYaacov2018}*{Section~2.3}, where this relation was considered in the classical, discrete setting, or to the more recent \cite{Iba2019}*{Section~1}, where the general metric case is discussed.

We recall some terminology. If $\phi(\bar u,\bar v)$ is a stable formula, $\bar a$ is a $\bar u$-tuple from a model $M\models T$, and $B\subseteq M$ is any subset, the \emph{$\phi$-canonical base} of $\bar a$ over $B$, denoted by $\Cb^M_\phi(\bar a/B)$, is the canonical parameter of the definition of the non-forking extensions of the $\phi$-type of $\bar a$ over $\acl^{M^\meq}(B)$. The $\phi$-canonical base is a \emph{metric imaginary element} of $M$; we denote by $M^\meq$ the expansion of $M$ by all the metric imaginary sorts. We always have $\Cb^M_\phi(\bar a/B)\in\acl^{M^\meq}(B)$. Then, given $\bar a,B,\bar c$ from $M$, we have that $\bar a\Findep[T]{B}\bar c$ (in words, \emph{$\bar a$ is stably independent from $\bar c$ over $B$}) iff for every stable formula $\phi$ we have $\Cb^M_\phi(\bar a/B\bar c)\in\acl^{M^\meq}(B)$, where $B\bar c$ is the union of $B$ and $\bar c$.

We also recall that non-forking extensions compatible with a given complete type always exist.

\begin{lemma}\label{l:indep-existence-prop}
For any small $\bar a,B,\bar c$ in a sufficiently saturated model $M\models T$, there exists $\bar a'$ in $M$ such that $\tp^M(\bar a'/B)=\tp^M(\bar a/B)$ and $\bar a'\Findep[T]{B} \bar c$.
\end{lemma}
\begin{proof}
See Ben Yaacov~\cite{BenYaacov2010b}*{Corollary~2.4}.
\end{proof}

If $T=\APA$ is the theory of atomless probability measure algebras, which is stable, the relation $\Findep[T]{}$ coincides with the usual notion of independence in probability theory, and we shall denote it simply by $\Findep{}$. We recall that if $\cA$ and $\cC$ are subalgebras of some ambient model of $\APA$ generated respectively by the tuples $\bar a$ and $\bar c$, and if $\cB$ is a subalgebra of $\cC$, then for any formula $\phi$ and $e \in \cA$, the $\phi$-canonical base of $e$ over $\cC$ is definable from $\E(\chi_e|\cC) \in L^\infty(\cC)$, and one has $\bar a\Findep{\cB}\bar c$ iff $\E(\chi_e|\cC)\in L^\infty(\cB)$ for every $e \in \cA$.

Given a $\Gamma$-system $\cX$, we will denote its complete first-order theory in the language of $\PMP_\Gamma$ by $T^\cX$. In the proof of the next lemma, we will also need to consider the reduct of $\cX$ to the language of $\APA$ (i.e., forget the action of $\Gamma$). We denote this reduct by $\cX|_\APA$. We also write $\cX|_\APA^\meq\coloneqq(\cX|_\APA)^\meq$. Note that the imaginary sorts of $\cX|_\APA$ are also imaginary sorts of $\cX$, so $\cX|_\APA^\meq$ is naturally included in $\cX^\meq$.

The following is an instance of a general fact about independence and reducts when the reduct enjoys weak elimination of imaginaries.

\begin{lemma}
  \label{l:Xmu-weak-elimination}
  Let $\cX$ be an atomless $\Gamma$-system, and let $\bar a,B,\bar c$ be taken from $\cX$. If $\bar a\Findep[T^\cX]{B} \bar c$, then $\bar a\Findep{\acl^\cX(B)}\bar c$.
\end{lemma}
\begin{proof}
Suppose $\bar a\Findep[T^\cX]{B} \bar c$, and fix a formula $\phi(\bar u,\bar v)$ in the language of $\APA$. Note that $\phi$ is stable. Let $e\coloneqq\Cb^\cX_\phi(\bar a/B\bar c)=\Cb^\cX_\phi(\bar a/\acl^\cX(B)\bar c)$ and $e'\coloneqq\Cb_\phi^{\cX|_\APA}(\bar a/\acl^\cX(B)\bar c)$. Observe that both $e$ and $e'$ belong to $\cX|_\APA^\meq$; in fact, up to fixing a uniform definition of $\phi$-types, we can see $e$ and $e'$ as elements of the same corresponding sort of canonical parameters. Since $\APA$ has weak elimination of imaginaries, there is a real tuple $\bar b\subset\cX$ such that $e\in\acl^{\cX|_\APA^\meq}(\bar b)$ and $\bar b\subset\acl^{\cX|_\APA^\meq}(e)$. On the other hand, by the independence hypothesis, $e\in\acl^{\cX^\meq}(B)$. Hence
$$\bar b\subset\acl^{\cX|_\APA^\meq}(e)\cap\cX\subseteq\acl^{\cX^\meq}(B)\cap\cX=\acl^\cX(B),$$
and thus $e\in\acl^{\cX|_\APA^\meq}(\acl^\cX(B))$. It follows that the extensions of $\phi$-types defined by $e$ are non-forking over $\acl^\cX(B)$ \big($\subseteq\acl^\cX(B)\bar c$\big) in the sense of $\cX|_\APA$. Hence $e'=e$, and in particular, $e'\in\acl^{\cX|_\APA^\meq}(\acl^\cX(B))$. We conclude that $\bar a\Findep{\acl^\cX(B)}\bar c$.
\end{proof}

Now recall that an extension $\cZ\subseteq\cX$ of $\Gamma$-systems is \emph{weakly mixing} if the relatively independent self-joining $\cX\otimes_\cZ \cX$ is ergodic.

\begin{theorem}\label{th:main-via-stability}
Let $\cZ\subseteq\cX$ be an extension of $\Gamma$-systems with $\cX$ strongly ergodic. Then the extension
$\acl^\cX(\cZ)\subseteq \cX$
is weakly mixing.
\end{theorem}
\begin{proof}
We denote $\cB = \acl^\cX(\cZ)$. If $\cX$ is not atomless, then by ergodicity it must be finite and, in that case, $\cB=\cX$; so we may assume $\cX$ is atomless. Let $\cX'$ be a sufficiently saturated elementary extension of $\cX$, which must be ergodic by the assumption of strong ergodicity of $\cX$. Let $\bar a$ be a tuple enumerating the set $\cX$. By Lemma~\ref{l:indep-existence-prop}, we can find a tuple $\bar a'$ in $\cX'$ with $\bar a' \equiv_\cB \bar a$ and $\bar a' \Findep[T^\cX]{\cB} \bar a$. Hence, by Lemma~\ref{l:Xmu-weak-elimination}, $\bar a' \Findep{\cB} \bar a$. The two conditions on $\bar a$ and $\bar a'$ say that the structure $\cW$ generated by $\bar a$ and $\bar a'$ inside $\cX'$ is isomorphic to the relatively independent self-joining $\cX\otimes_\cB \cX$. But then, since $\cW\subseteq\cX'$, we conclude that $\cX\otimes_\cB \cX$ is ergodic as required.
\end{proof}

We point out that Theorem~\ref{th:distal-strongly-erg-acl} is a direct consequence of this result, modulo the following simple remark.

\begin{lemma} Let $\cX$ be an ergodic $\Gamma$-system and let $\cZ\subseteq \cY\subseteq\cX$ be factors. If the extension $\cY\subseteq\cX$ is weakly mixing, then $\cY$ contains the maximal distal extension of $\cZ$ in~$\cX$.
\end{lemma}
\begin{proof}
By induction, it suffices to show that $\cY$ contains $\cZ_1$, the maximal compact extension of $\cZ$ in $\cX$. Let $\cW$ be the factor generated by $\cY$ and $\cZ_1$. Then the extension $\cY\subseteq\cW$ is clearly weakly mixing and it is also compact by Proposition~\ref{p:factor-generated-by-compact-factors}. So it is enough to note that a compact, weakly mixing extension must be trivial or, which is the same, that a weakly mixing extension $\cY\subseteq\cW$ does not admit any finitely generated $\Gamma$-invariant submodule. Indeed, if $e_0,\dots,e_{n-1}$ is a $\cY$-basis for a non-trivial invariant $\cY$-module in $L^2(\cW)$, then the function $\sum_{i<n}e_i\otimes \conj{e_i} \in L^2(\cW\otimes_\cY \cW)$ is $\Gamma$-invariant but it is not constant.
\end{proof}

We thus see that one can obtain Theorem~\ref{th:distal-strongly-erg-acl} and its various ergodic-theoretic consequences by almost purely model-theoretic means. We may point out that this is also true of our result about coalescence, Theorem~\ref{th:distal-coalescent}. Indeed, using a somewhat lengthy (but elementary and purely model-theoretic) compactness argument, which we omit, one can improve the conclusion of Theorem~\ref{th:distal-strongly-erg-acl} to say that the maximal distal extension of $\cZ$ in $\cX$ is contained in $\Eacl^\cX(\cZ)$, the \emph{existential algebraic closure} of $\cZ$. The set $\Eacl^\cX(\cZ)$ is the union of all compact subsets $A \sub \cX$ \df{existentially definable} over $\cZ$, i.e., such that the distance function $d(\cdot, A)$ is a uniform limit of $\inf$-formulas with parameters from $\cZ$. (This notion has the feature of inducing a closure operator $\Eacl \colon\mathcal{P}(\cX)\to\mathcal{P}(\cX)$, with properties similar to the ones of the usual algebraic closure.) Since every $\cZ$-embedding $\cX\to\cX$ must restrict to an isometry on every existentially $\cZ$-definable compact set, the condition $\cX=\Eacl^\cX(\cZ)$ readily implies that $\cX$ is coalescent over $\cZ$, and Theorem~\ref{th:distal-coalescent} follows. This alternative route (using compactness to pass from $\acl$ to $\acl_\exists$) was our original approach to the coalescence result, before we realized that we could prove the stronger Theorem~\ref{th:main-lemma}; the proof of the latter requires more ergodic theory, but shows that algebraicity of strongly ergodic distal factors can be witnessed by \emph{quantifier-free} formulas.

%%%%%%%%%%%%%%%%%%%%%%%%%%%%%%%%%%%%%%%%%%%%%%%%%%

\section{Groups with property (T)}
\label{sec:groups-with-property-T}

Next, we consider actions of groups with property (T). Here we will be mostly interested in separable systems because, by Theorem~\ref{th:distal-strongly-erg-acl}, if $\cZ$ is a separable system of a property (T) group (for example, the trivial one), then any distal extension of $\cZ$ is also separable. Thus we will make free use of the classical theory for separable systems, as presented, for example, in \cite{Glasner2003}. Nevertheless, we state Theorem~\ref{th:propT} in full generality and prove it by a reduction to the separable case.

It will be useful to consider the representation of ergodic compact extensions as skew-products. More precisely, recall that if $\cZ \sub \cX$ is an ergodic compact extension with $\cX$ separable, then there exists a compact, metrizable group $K$, a closed subgroup $L \leq K$, and a cocycle $\rho \colon \Gamma \times Z \to K$ such that
\begin{equation*}
  X \cong Z \times_\rho K/L
\end{equation*}
with the action of $\Gamma$ given by $\gamma\cdot (z,kL) = {(\gamma\cdot z,\rho(\gamma,z)kL)}$. Moreover, $K$ can be chosen so that the system $X \cong Z \times_\rho K$ is also an ergodic compact extension of $\cZ$.

Our result about property (T) groups follows from the following more general theorem.
\begin{theorem}
  \label{th:propT}
Let $\cZ$ be a $\Gamma$-system and suppose that $\cD_2(\cZ)$ is strongly ergodic. Then $\cD_2(\cZ) = \cD_1(\cZ)$. In particular, every distal extension of $\cZ$ is a compact extension.
\end{theorem}
\begin{proof}
We write $\cD_i$ instead of $\cD_i(\cZ)$. We assume first that $\cZ$ is separable. As noted above, by Theorem~\ref{th:distal-strongly-erg-acl}, both systems $\cD_1$ and $\cD_2$ are separable. We represent $\cD_1$ and $\cZ$ as usual actions on probability spaces, by $\Gamma\actson D_1$ and $\Gamma\actson Z$, respectively.

Hence we can write $D_1 \cong Z\times_{\rho_1} K/L$ for appropriate $K$, $L$ and $\rho_1$ as above. In fact, we must have $L = \set{1_{K}}$, otherwise the skew-product $Z\times_{\rho_1} K$ would be a proper extension of $\cD_1$ which is still an ergodic compact extension of $\cZ$, contradicting the maximality of $\cD_1$. Thus $D_1 \cong Z\times_{\rho_1} K$. Note that $K$ embeds as a subgroup of $\Aut_\cZ(\cD_1)$ by acting on $Z \times_\rho K$ on the right. Moreover, if we denote
\begin{equation*}
  \cD_1^K = \set{a \in \cD_1 : k \cdot a = a \text{ for all } k \in K}, 
\end{equation*}
then, clearly, $\cD_1^K = \cZ$. (We actually have $K=\Aut_\cZ(\cD_1)$, see \cite{Glasner2003}*{Proposition~6.16}.)

Now let $K_2 = \Aut_\cZ(\cD_2)$, and consider the corresponding factor of invariant elements $\cD_2^{K_2} \subseteq \cD_2$. We note that as by Corollary~\ref{c:autZ-to-authatZ}, the restriction map $\Aut_\cZ(\cD_2) \to \Aut_\cZ(\cD_1)$ is surjective, we have $\cD_1 \cap \cD_2^{K_2} = \cD_1^K = \cZ$.
Let $\cC$ be the maximal compact extension of $\cZ$ in $\cD_2^{K_2}$. Then $\cC \sub \cD_1$, for $\cD_1$ is the maximal intermediate compact extension of $\cZ$ within $\cD_2$. Thus $\cC \sub \cZ$. As the extension $\cZ \sub \cD_2^{K_2}$ is distal (as an intermediate extension of the distal extension $\cZ \sub \cD_2$), this implies that $\cZ = \cD_2^{K_2}$.

Finally, by our hypothesis on $\cD_2$ and Corollary~\ref{c:Aut(distal)-is-compact}, the group $K_2$ is compact, i.e., $\cZ = \cD_2^{K_2} \sub \cD_2$ is a \df{group extension}. By \cite{Glasner2003}*{Theorems~3.29 and~9.14}, it must be a compact extension, so $\cD_2 = \cD_1$. We conclude that $\cD_\alpha = \cD_1$ for all $\alpha$. By Theorem~\ref{th:universal-distal-systems}, this means that every distal extension of $\cZ$ is a factor of $\cD_1$, and thus a compact extension of $\cZ$ as per Lemma~\ref{l:factor-of-compact-is-compact}. 

Now suppose $\cZ$ is an arbitrary ergodic system with $\cD_2(\cZ)$ strongly ergodic. Let $\cY\subseteq \cD_2$ be any separable factor. By Lemma~\ref{l:transfer-lemma}, there are separable factors $\cZ'\subseteq \cY'\subseteq \cD_2$ with $\cZ'\sub \cZ$ and $\cY\subseteq\cY'$ such that the extension $\cZ'\subseteq\cY'$ is distal. As $\cD_2(\cZ')$ is a factor of $\cD_2(\cZ)$ (Proposition~\ref{p:inclusion-Dalpha}), it is also strongly ergodic and by what we have just shown in the separable case, the extension $\cZ' \sub \cY'$ is compact. Hence the system generated by $\cY$ and $\cZ$ is a compact extension of $\cZ$, and by Proposition~\ref{p:factor-generated-by-compact-factors}, we can conclude that so is $\cD_2$. So again $\cD_2 = \cD_1$ and we deduce as before that every ergodic distal extension of $\cZ$ is actually a compact extension.
\end{proof}

\begin{cor}
  \label{c:T-distal-implies-compact}
  Suppose that $\Gamma$ has property (T). Then the following hold:
  \begin{enumerate}
  \item If $\cZ$ is any ergodic $\Gamma$-system, then every distal extension of $\cZ$ is a compact extension.
  \item \label{i:cTdic:2} Every ergodic, distal $\Gamma$-system is compact.
  \end{enumerate}
\end{cor}

If one is just interested in item \ref{i:cTdic:2} of the last corollary, one can also obtain it as a consequence of Corollaries \ref{c:from-Veech} and \ref{c:distal-is-in-Xalg}.

A variant of the previous argument in the case $\cZ=1$ (see below) shows that the conclusion $\cD_2(\Gamma)=\cD_1(\Gamma)$ can be obtained by assuming only that $\Aut(\cD_2(\Gamma))$ is compact. On the other hand, note that in the exact sequence
\begin{equation}
  \label{eq:exact-sequence}
  1 \to \Aut_{\cD_1(\Gamma)}(\cD_2(\Gamma)) \to \Aut(\cD_2(\Gamma)) \to \Aut(\cD_1(\Gamma)) \to 1
\end{equation}
both groups to the sides of $\Aut(\cD_2(\Gamma))$ are always compact. If $\Aut(\cD_2(\Gamma))$ is Polish, this implies that $\Aut(\cD_2(\Gamma))$ is also compact. Therefore we obtain the following strengthening of Corollary~\ref{c:from-Veech}.

\begin{theorem}\label{th:polish-AutD2}
Suppose the group $\Aut(\cD_2(\Gamma))$ is either compact or Polish. Then, $\cD_2(\Gamma)=\cD_1(\Gamma)$.
\end{theorem}
\begin{proof}
We write $\cD_i$ instead of $\cD_i(\Gamma)$ and $K_i$ instead of $\Aut(\cD_i)$. In either case of the statement we have that $K_2$ is compact. Indeed, $K_1$ is clearly compact and $\Aut_{\cD_1}(\cD_2)$ is compact by Proposition~\ref{p:compact-aut-group}. If $K_2$ is Polish, then the map $K_2 \to K_1$ in \eqref{eq:exact-sequence} must be open and we can apply \cite{Hewitt1979}*{Theorem~5.25} to conclude that $K_2$ is compact.

Now, as in the proof of Theorem~\ref{th:propT}, the surjectivity of the restriction map $K_2\to K_1$ gives $\cD_1\cap\cD_2^{K_2}=\cD_1^{K_1}=1$, hence the factor of invariant elements $\cD_2^{K_2}$ must be trivial. In other words, the action $K_2\actson \cD_2$ is ergodic. As in the proof of Corollary~\ref{c:Robin}, this implies that $\cD_2$ is compact, as required.
\end{proof}

\begin{remark}
  \label{rem:D1-separable}
  In view of Theorem~\ref{th:polish-AutD2}, one might ask whether it is enough to assume that $\cD_1(\Gamma)$ is separable in order to conclude that $\cD_2(\Gamma) = \cD_1(\Gamma)$. The answer is negative, as shown by the following example. Let $\Gamma$ be the Grigorchuk group (see \cite{Harpe2000} for details on this group). First, any unitary representation of $\Gamma$ factors through a finite group \cite{Harpe2000}*{p.~224}. On the other hand, $\Gamma$ is residually finite and finitely generated, so its profinite completion, which by the previous remark is equal to its Bohr compactification, is an infinite, metrizable, compact group. In other words, the probability space of $\cD_1(\Gamma)$ is standard and non-atomic. The group $\Gamma$ is amenable, so by \cite{Ornstein1980}, the action $\Gamma \actson \cD_1(\Gamma)$ is orbit equivalent to an action of $\Z$, and thus, by Zimmer~\cite{Zimmer1977}, there exist non-trivial cocycles $\Gamma \times \cD_1(\Gamma) \to K$ for any compact, metrizable $K$. Hence $\cD_1(\Gamma) \subsetneq \cD_2(\Gamma)$
\end{remark}

\appendix

\section{Distal rank and factors}
\label{sec:appendix}

It is a folklore result that the distal rank of a factor cannot be larger than the distal rank of the original system, at least for actions of $\Z$. We have used this to deduce Theorem~\ref{th:rank-of-Domega1Z}. Since we could not find a proof in the literature, we include one here. We thank Eli Glasner, Matt Foreman and Jean-Paul Thouvenot for some helpful hints.

We use again the convention that if $\cX$ and $\cY$ are factors of a system $\cW$, then $\cX\cY$ denotes the factor generated by them inside $\cW$. We also recall that we write $\cX\Findep{\cZ} \cY$ to say that $\cX$ and $\cY$ are relatively independent over $\cZ$.

\begin{prop}\label{p:Furst:max-comp-factor-of-a-product}
Let $\cX$ and $\cY$ be separable $\Gamma$-systems inside some larger system, with a common factor $\cZ$. Assume $\cX$ is ergodic and $\cX\Findep{\cZ} \cY$. Let $\cY_1$ be the maximal compact extension of $\cZ$ in $\cY$. Then $\cX\cY_1$ is the maximal compact extension of $\cX$ inside $\cX\cY$.
\end{prop}
\begin{proof}
For $\Z$-actions this is exactly Theorem~7.4 in Furstenberg's article \cite{Furstenberg1977}, and its proof also works for actions of arbitrary countable groups. See also Glasner's book~\cite{Glasner2003} (Chapter~9, Sections 3 and 4, and particularly Theorem~9.21), where a variant of this theorem is proved for arbitrary $\Gamma$-actions along the same lines as in \cite{Furstenberg1977} (this variant corresponds precisely to Theorem~7.1 in~\cite{Furstenberg1977}).
\end{proof}

\begin{prop}\label{p:rank-of-factors}
  Suppose $\Gamma$ is a countable group and $\cZ\subseteq \cY\subseteq \cX$ are $\Gamma$-systems with $\cX$ ergodic and the extensions $\cZ \sub \cY$ and $\cY \sub \cX$ distal. Then $\rk(\cY/\cZ)\leq \rk(\cX/\cZ)$.
\end{prop}
\begin{proof}
If this holds in the separable case, then the general case can be deduced from it exactly as in the proof of Proposition~\ref{p:factor-of-distal-is-distal} (note that the bounds for the distal rank that appear in Lemma~\ref{l:transfer-lemma} and Proposition~\ref{p:factor-generated-by-distal-factors} then become relevant).

So we assume $\cX$ is separable. For each ordinal $\eta$, let $\cX_\eta$ and $\cY_\eta$ be the maximal distal intermediate extensions of $\cZ$ of rank at most $\eta$ inside $\cX$ and $\cY$, respectively. If $\alpha=\rk(\cX/\cZ)$ and $\beta=\rk(\cY/\cZ)$, then $(\cX_\eta)_{\eta\leq\alpha}$ and $(\cY_\eta)_{\eta\leq\beta}$ are the corresponding distal towers for $\cX$ and $\cY$ over $\cZ$ and each $\cX_{\eta+1}$ is the maximal compact extension of $\cX_\eta$ inside $\cX$, and similarly for $\cY_{\eta+1}$.

We observe first that by the maximality of the $\cX_\eta$, we have $\cY_\eta\subseteq \cX_\eta$ for every~$\eta$. Next we show by induction that $\cX_\eta \Findep{\cY_\eta} \cY$ for each $\eta$. Thus, in particular, $\cX\Findep{\cY_\alpha} \cY$, which implies $\cY_\alpha=\cY$ and hence $\beta\leq \alpha$, as required.

For $\eta=0$, the claim becomes $\cZ\Findep{\cZ} \cY$ and holds trivially. Suppose inductively that $\cX_\eta \Findep{\cY_\eta} \cY$. By Proposition~\ref{p:Furst:max-comp-factor-of-a-product}, $\cX_\eta\cY_{\eta+1}$ is the maximal compact extension of $\cX_\eta$ inside $\cX_\eta\cY$. On the other hand, the projection to $\cX_\eta\cY$ of any finitely generated $\Gamma$-invariant module over $\cX_\eta$ is again a finitely generated $\Gamma$-invariant module over $\cX_\eta$. It follows that the projection of $L^2(\cX_{\eta+1})$ to $L^2(\cX_\eta\cY)$ is contained in $L^2(\cX_\eta\cY_{\eta+1})$ or, in other words, that
$$\cX_{\eta+1} \Findep{\cX_\eta \cY_{\eta+1}}  \cX_\eta\cY.$$
This together with the induction hypothesis implies that $\cX_{\eta+1} \Findep{\cY_{\eta+1}} \cY$. Finally, suppose $\eta$ is a limit ordinal and $\cX_\xi \Findep{\cY_\xi} \cY$ for every $\xi<\eta$. We can write any $f\in L^2(\cX_\eta)$ as $f=\sum_{\xi<\eta}f_\xi$ for certain $f_\xi\in L^2(\cX_\xi)$. Hence,
$$\E(f|\cY)=\sum_{\xi<\eta} \E(f_\xi|\cY)= \sum_{\xi<\eta} \E(f_\xi|\cY_\xi) \in L^2(\cY_\eta).$$
This shows that $\cX_\eta \Findep{\cY_\eta} \cY$ and concludes the proof.
\end{proof}

\bibliography{st-distal}
\end{document}